\documentclass[11pt,reqno]{amsart}
\usepackage{amsmath,amsthm,amsfonts,amssymb,amscd,amsbsy}
\usepackage[utf8]{inputenc}
\usepackage{array,colortbl}
\usepackage{tikz}
\usetikzlibrary{cd,arrows,babel}
\usepackage{stmaryrd}
\usepackage{mathtools}
\usepackage{graphicx}
\usepackage[english]{babel}
\usepackage{fullpage}
\usepackage{comment}
\usepackage{relsize}

\usepackage{graphicx}
\usepackage[colorlinks=true, allcolors=blue]{hyperref}

\numberwithin{equation}{section}

\newcommand{\N}{\mathbb{N}}
\newcommand{\ideal}{\mathcal{I}}
\newcommand{\idealj}{\mathcal{J}}

\newcommand\fin{{\sf Fin}}

\newcommand{\seq}{{\N}^{<\omega}}

\newtheorem{theorem}{Theorem}[section]
\newtheorem{lemma}[theorem]{Lemma}

\newtheorem{proposition}[theorem]{Proposition}
\newtheorem{corollary}[theorem]{Corollary}
\newtheorem{claim}[theorem]{Claim}

\theoremstyle{definition}
\newtheorem{remark}[theorem]{Remark}

\newtheorem{definition}[theorem]{Definition}
\newtheorem{example}[theorem]{Example}

\title{On the complementation of spaces of $\mathcal I$-null sequences}
\author{Michael A. Rincón-Villamizar, Carlos  Uzcátegui-Aylwin}

\begin{document}

\begin{abstract}
 We study the complementation (in $\ell_\infty$) of the Banach space $c_{0,\mathcal{I}}$, consisting of all bounded sequences $(x_n)$ that $\mathcal{I}$-converge to $0$, endowed with the supremum norm, where $\mathcal{I}$ is an ideal of subsets of $\mathbb{N}$. We show that the complementation of these spaces is related to a condition requiring that the ideal is the intersection of at most a countable family of maximal ideals, which we refer to as at most $\omega$-maximal ideals. We prove that $\mathcal{I}$  is at most $\omega$-maximal exactly when $c_{0,\mathcal{I}}$ is the kernel of an operator from $\ell_\infty$ to itself satisfying a certain property. 
In addition, we show that the existence of a Banach lattice isomorphism from the quotient $\ell_\infty/c_{0,\mathcal{I}}$ onto a closed sublattice of $\ell_\infty$ is equivalent to $\mathcal{I}$ being an at most $\omega$-maximal ideal. Moreover, $\ell_\infty/c_{0,\mathcal I}$ is Banach lattice isomorphic to $\ell_\infty$ if and only if
$\mathcal I$ is the intersection of a certain countably infinite family of maximal ideals; we refer to such ideals as strongly $\omega$-maximal. 

Finally, for two ideals $\mathcal I\subsetneq\mathcal J$, we characterize when the quotient space $c_{0,\mathcal{J}} / c_{0,\mathcal{I}}$ is finite-dimensional.
\end{abstract}
\maketitle

\section{Introduction}
An ideal on $\mathbb{N}$ is a collection $\ideal$  of subsets of $\mathbb{N}$ closed under finite unions and taking subsets of its elements. A sequence $(x_n)$ in a Banach space $X$ is said to be $\mathcal{I}$-convergent to $x\in X$, denoted as $\mathcal{I}$-$\lim x_n=x$, if for each $\varepsilon>0$, the set $\{n\in\mathbb{N}\,\colon\,\|x_n-x\|\geq\varepsilon\}$ belongs to $\mathcal{I}$. When $\mathcal{I}$ is $\fin$, the ideal of finite subsets of $\mathbb{N}$, we have the classical convergence in $X$. For this reason, it is natural—and we will adopt this assumption—to require that $\fin$ is contained in every ideal under consideration.  The $\mathcal{I}$-convergence was introduced in \cite{k-s-w}, although many authors had already studied this concept in particular cases and in different contexts (see, for instance, \cite{balcerzak,b-s-w,fast,filipow,k-m-s}).  We are interested in the following space
\begin{gather*}
c_{0,\mathcal{I}}=\{(x_n)\in\ell_\infty\,\colon\,\mathcal{I}-\lim x_n=0\}.
\end{gather*}
We showed in \cite{rincon-uzcategui} that $c_{0,\mathcal{I}}$ is a closed subspace of $\ell_\infty$, and that some of its Banach and Banach-lattice properties are closely related to the combinatorial and topological properties of the ideal $\mathcal{I}$. For instance, a 
closed sublattice of $\ell_\infty$ is an ideal exactly when it is of the form $c_{0,\ideal}$ for some ideal  $\ideal$  on $\N$. Furthermore,  $c_{0,\mathcal I}$ and $c_{0,\mathcal J}$ are isometric if, and only if,  $\mathcal I$ and $\mathcal J$ are isomorphic.
The main objective of this paper is to investigate the phenomenon of complementation of $c_{0,\mathcal{I}}$ in $\ell_\infty$. Some results in this direction are already known.

We call a proper  ideal $\ideal$  \textit{complemented} if $c_{0,\ideal}$ is complemented in $\ell_\infty$.  It is a classical result that $c_0$ is not complemented in $\ell_\infty$, so $\fin$ is not a complemented ideal. 
There has been some recent interest in complemented ideals \cite{hrusak-saenz2025,  kania,Leonetti2018}. Leonetti \cite{Leonetti2018} proved that no meager ideal is  complemented, where a meager ideal refers to an ideal that is meager as a subset of the Cantor cube $\{0,1\}^\N$, identified via characteristic functions. The key element of his argument is the existence of an uncountable family of subsets in $\mathcal{P}(\mathbb{N}) \setminus \mathcal{I}$ such that $A\cap B\in \ideal$ for any  distinct $A, B$ in the family (the so-called $\mathcal{I}$-AD families).   We show that if $\ideal$ is complemented, then any $\ideal$-AD family is at most countable and therefore $\ideal$ is not meager and  does not have the Baire property  as a subset of   $\{0,1 \}^\N$.

On the other hand, Kania \cite{kania} observed that the intersection of a finite collection of maximal ideals is complemented (for a proof, see \cite{rincon-uzcategui}).
We call an ideal {\em $\omega$-maximal} if it can be written as an intersection of a countably infinite family  of maximal ideals.
We extend Kania's result to certain special $\omega$-maximal ideals. Finally, continuing this approach, Hru\v{s}\'ak-S\'aenz \cite{hrusak-saenz2025} found very recently a characterization of complemented ideals in terms of the quotient $\ell_\infty/ c_{0,\ideal}$. Namely, $\ideal$ is complemented if and only if the quotient $\ell_\infty/ c_{0,\ideal}$ is isomorphic to a closed subspace of $\ell_\infty$. 

Motivated by the discussion above, our work focuses on analyzing the structural properties of complemented ideals.

We say that an ideal is {\em strongly $\omega$-maximal} if there exists a countably infinite collection
$\{\mathcal{I}_n \colon n \in \mathbb{N} \}$ of maximal ideals such that $\mathcal{I} = \bigcap_n \mathcal{I}_n$, and the family $\{\mathcal{I}_n^* \colon n \in \mathbb{N} \}$ is discrete in $\beta\mathbb{N}$ (where $\mathcal{I}^*$ denotes the dual filter and $\beta\mathbb{N}$ is the Stone–Čech compactification of $\mathbb{N}$).
We provide an example of an $\omega$-maximal ideal that is not strongly $\omega$-maximal. Since $\omega$-maximal ideals play a central role in our results, we present additional properties about them in Section \ref{kappa-max}. The dual notion of a filter  represented as an intersection of a finite or countable family of ultrafilters has been recently studied by Bergman \cite{Bergman2014} and Kadets,  Seliutin,  and Tryba \cite{kadetsetal2022}. In particular, we show that the notion introduced in \cite{kadetsetal2022} of a filter admitting a minimal countable representation corresponds to our notion of strongly $\omega$-maximal ideal (see Remark \ref{minimal-repre}).

In Section \ref{dimension-quotient}, we study the quotient $c_{0,\mathcal{J}}/c_{0,\mathcal{I}}$ when $\mathcal{I} \subsetneq \mathcal{J}$ are ideals. We provide a combinatorial characterization of the finite-dimensionality of $c_{0,\mathcal{J}}/c_{0,\mathcal{I}}$.
In particular, we obtain that $\ell_\infty/c_{0,\mathcal{I}}$ is finite-dimensional if and only if $\mathcal{I}$ is a finite intersection of maximal ideals.

In Section \ref{complementacion}, we discuss properties of complemented ideals. We show that the complementation of ideals is stable under sums, countable intersections and restrictions. So, in particular, every (strongly) $\omega$-maximal ideal is complemented. 

In Section \ref{at most maximal}, we also prove that any strongly $\omega$-maximal ideal is complemented by showing an explicit projection (see Theorem \ref{proyecciones1}). Furthermore, we characterize strongly $\omega$-maximal ideals in terms of the type of projections on their associated space $c_{0,\mathcal{I}}$ (see Theorem \ref{proyecciones1b}). Finally, we show that $\mathcal I$ is a strongly $\omega$-maximal ideal if and only if the quotient $\ell_\infty/c_{0,\mathcal I}$ is Banach lattice isomorphic to $\ell_\infty$ (see Theorem \ref{s-omega-max caracterizacion in terms of the cociente}). 

An ideal $\ideal$ is called {\em at most $\omega$-maximal} if it is an intersection of at most a countable family of maximal ideals. We show that the existence of a Banach lattice isomorphism from $\ell_\infty/c_{0,\mathcal{I}}$ onto a closed sublattice of $\ell_\infty$ is equivalent to $\mathcal{I}$ being at most $\omega$-maximal ideal (see Theorem \ref{implicaciones}). 

Finally, we show that $c_0$ is not complemented in $c_{0,\mathcal{J}}$ for any ideal $\mathcal{J}$ properly extending $\mathsf{Fin}$. This property is shared by all ideals $\mathcal{I}$ such that $\mathcal{I} \restriction A$ is Baire measurable on $2^A$ for any $A \notin \mathcal{I}$. For instance, all analytic ideals have this hereditary property. Also, we present several examples of ideals which are not complemented.

\section{Preliminaries}
\label{preliminares}

We will use standard terminology and notation for Banach lattices and Banach space theory. For unexplained definitions and notations, we refer to \cite{AK,schaeffer}. The scalar field is denoted by $\mathbb K$. All Banach lattices analyzed here are assumed to be real. However, our results can be extended to complex Banach lattices in the usual manner \cite[Chapter 2, p. 133]{schaeffer}. If $X$ and $Y$ are isomorphic (isometric) Banach spaces, we write $X\sim Y$ ($X\cong Y$, respectively). If $E$ is a closed subspace of a Banach space $X$, we say that 
$E$ is complemented in $X$ if there is a continuous onto operator $P\colon X\to E$ such that $P^2=P$, or equivalently, there is a closed subspace $W$ of $X$ such that $X=E\oplus W$. In addition, if $E$ and $X$ are Banach lattices, $P$ is called \textit{positive} if $Px\geq{\bf0}$ for all $x\geq{\bf0}$. Also, a bounded linear operator $T\colon E\to X$ is called \textit{Banach lattice homomorphism} if $T(x\vee y)=Tx\vee Ty$ for each $x,y\in E$.

An ideal $\mathcal I$ on 
 a set $X$ is a collection of subsets of $X$ satisfying:
\begin{enumerate}
    \item $\emptyset\in\mathcal I$;
    \item If $A\subseteq B$ and $B\in\mathcal I$, then $A\in\mathcal I$;
    \item If $A,B\in\mathcal I$, then $A\cup B\in\mathcal I$.
\end{enumerate}

We always assume that every finite subset of $X$ belongs to $\mathcal{I}$. The dual filter of an ideal $\mathcal{I}$ is denoted by $\mathcal{I}^*$ and consists of all sets of the form $X\setminus A$ for some $A\in \mathcal{I}$. Conversely, if $\mathcal{F}$ is a filter, then $\mathcal{F}^*$ denotes its dual ideal which consists of all sets of the form $X\setminus A$ for some $A\in \mathcal{F}$. The {\em co-ideal} $\mathcal{I}^+$ is the collection $\mathcal{P}(X)\setminus \mathcal{I}$.  

When $X$ is countable, an ideal $\mathcal{I}$ can be conveniently seen as a subset of the Cantor cube $\{0,1\}^X$ with the compact metric topology. This allows us to consider when $\mathcal{I}$ is a meager subset of the Cantor cube. The following is a very useful result:

\begin{theorem}[Jalali-Naini, Talagrand \cite{jalali,talagrand}]
\label{talagrand theorem}
Let $\mathcal I$ be a proper ideal on $\mathbb N$. The following statements are equivalent:
    \begin{enumerate}
        \item $\mathcal I$ is meager.
        \item $\mathcal I$ has the Baire property.
        \item There is a partition $\{F_k\,\colon\,k\in\mathbb N\}$ of $\mathbb N$ into finite sets such that for every $M\subseteq\mathbb N$ infinite we have $\bigcup_{k\in M}F_k\not\in\mathcal I$.
    \end{enumerate}
\end{theorem}

An ideal $\ideal$ is {\em maximal} if $\mathcal P(X)$ is the only ideal properly  extending  $\ideal$; equivalently, if $\ideal^*$ is an ultrafilter. Notice that $\ideal$ is maximal if $\ideal^*=\ideal^+$. For $A\subseteq X$, we denote the restriction of $\ideal$ to $A$ by $\mathcal I\restriction A=\{A\cap B\,\colon B\in\mathcal I\}$ which is an ideal on $A$. If $\mathcal A$ and $\mathcal B$ are  families of sets, we denote by $\mathcal A\sqcup\mathcal B$ the collection $\{A\cup B\,\colon\,A\in\mathcal A,B\in\mathcal B\}.$
An $\mathcal I$-AD family is a collection $\mathcal{A}\subseteq \ideal^+$ such that $A\cap B\in \ideal$ for every two different sets $A, B\in \mathcal{A}$.  Two ideals $\ideal$ and $\idealj$ on $X$ and $Y$, respectively, are {\em isomorphic}, if there is a bijection $f:X\to Y$ such that $f[E]\in \idealj$ for all $E\in \ideal$. 

The  following  observations  will be needed in the sequel; its proof is straightforward.

\begin{lemma}
\label{condicion equivalente a maximalidad}
Let $\mathcal I$ be an ideal on $\mathbb N$ and $A\in\mathcal I^+$. 

\begin{enumerate}
\item If $\mathcal I\restriction A$ is a maximal ideal on $A$, then $\mathcal J=\mathcal I\restriction A\sqcup\mathcal P(A^c)$ is a maximal ideal on $\mathbb N$.

\item  If $\mathcal I$ is maximal and $A\in\mathcal I^+$, then $\mathcal I=\mathcal I\restriction A\sqcup\mathcal P(A^c)$.

\item  The following assertions are equivalent:

\begin{enumerate}
\item $\mathcal I\restriction A$ is maximal on $A$;

\item Let $\mathcal J_A=\mathcal I\sqcup\mathcal P(A)$. Then $\mathcal I\restriction B$ is maximal on $B$ for all $B\in\mathcal J_A\setminus\mathcal I$.
\end{enumerate}
\end{enumerate}
\end{lemma}

Let $\{K_n:\; n\in F\}$ be a partition of a countable set $X$, where $F\subseteq \N$. For $n\in F$, let $\ideal_n$ be an ideal on $K_n$. The direct sum, denoted by $\bigoplus\limits_{n\in F} \ideal_n$, is defined as follows:
\[
A\in \bigoplus_{n\in F} \ideal_n \Leftrightarrow (\forall n\in F)(A\cap K_n \in \ideal_n).
\]
Notice that the direct sum $\bigoplus_n \ideal_n$ can also be  naturally  defined in $\N\times\N$.  When $\mathcal I_n$ is isomorphic to $\mathcal I$  for all $n$, the direct sum is denoted by $\mathcal I^\omega$.

If  $\ideal$ is a maximal ideal on $\N$ and $K\not\in \ideal$, it is easy to see that $A\in \ideal$ if and only if $A\cap K\in \ideal$. From this, we have the following observation that will be used later on. 

\begin{lemma}
\label{suma-directa}
Let $\{\mathcal I_n\,\colon\,n\in F\}$ be a countable collection of maximal ideals on $\N$ and  $\{K_n\,\colon\,n\in F\}$ be a family of pairwise disjoint subsets of $\N$ such that $K_n \not\in\mathcal I_n$ for each $n\in F$. Then, $A\in \bigcap_{n\in F}\ideal_n$ if and only if $A\cap K_n\in\mathcal I_n$ for every $n\in F$. In particular, if 
$\{K_n\,\colon\,n\in F\}$ is a partition of $\mathbb N$ such that $K_n \not\in\mathcal I_n$ for each $n\in F$, then
\[
\bigoplus_{n\in F} (\ideal_n\restriction K_n) \,=\, \bigcap_{n\in F}\ideal_n.
\]
\end{lemma}

Recall that $\beta\N$ is the Stone-\v{C}ech compactification of $\N$ which  is usually identified with the collection of all ultrafilters on $\N$. For a set $A\subseteq\mathbb N$, we let $A^*=\{p\in\beta\mathbb N\,\colon\,A\in p\}$.  The family $\{A^*\,\colon\,A\subseteq\mathbb N\}$ defines a basis for the topology of $\beta\mathbb N$. As usual, we identify each $n\in \N$ with the principal ultrafilter $\{A\subseteq\N: n\in A\}$. Notice that every principal ultrafilter is an isolated point of $\beta\N$.

If ${\bf x}=(x_n)\in\ell_\infty$ and $\varepsilon>0$, we will use throughout the whole paper the following notation:
\begin{gather*}
    A(\varepsilon,{\bf x})=\{n\in\mathbb N\,\colon\,|x_n|\geq\varepsilon\}.
\end{gather*}
Notice that ${\bf x}\in c_{0,\mathcal{I}}$ if and only if $A(\varepsilon,{\bf x})\in \ideal$ for all $\varepsilon>0$. 

\section{\texorpdfstring{$\omega$}{omega}-maximal ideals}
\label{kappa-max}

Every ideal on $\N$ is easily seen to be equal to an  intersection of a collection of $2^{\aleph_0}$ many  maximal ideals.  Therefore,  an ideal $\ideal$ on $\N$  is called {\em $\kappa$-maximal}, for $\kappa$ a cardinal, if $\ideal=\bigcap_{\alpha<\kappa}\ideal_\alpha$ for some  maximal ideals $\ideal_\alpha$, for $\alpha<\kappa$ and $\kappa$  has the smallest possible value. Clearly, the interesting cases  are the  $\kappa$-maximal ideals with  $\kappa<2^{\aleph_0}$.

Since we are always assuming  that $\fin$ is contained in any ideal under consideration, every maximal ideal extending a given ideal is necessarily non-principal.  It is worth keeping in mind that the intersection of less than $2^{\aleph_0}$ maximal ideals on $\N$ does not have the Baire property, and hence it is non-meager (\cite{plewik}, \cite[Proposition 23]{talagrand}). Our main interest will be on $\omega$-maximal ideals since they are related to the  complementation of $c_{0,\mathcal I}$ in $\ell_\infty$.

We say that a maximal ideal $\idealj$ is a {\em limit point} of a set $\mathfrak D$ of maximal ideals, if $\idealj^*$ is a limit point of $\mathfrak{D}^*\coloneqq\{\ideal^*:\ideal \in \mathfrak{D}\}$ in $\beta\N$.  We say that a countable collection of maximal ideals $\{\ideal_n: n\in \N\}$ is {\em discrete} if  $\{\ideal^*_n: n\in \N\}$ is a discrete subset of $\beta\N$. An ideal $\ideal$ is {\em strongly $\omega$-maximal} if  there is a discrete infinite collection $\{\mathcal{I}_n:n\in \N\}$  of maximal ideals on $\N$ such that $\ideal=\bigcap_n \mathcal{I}_n$.

Now, we present a useful characterization of discrete families of maximal ideals. The case of finite families  was already shown by A. Mill\'an \cite{millan}  and Bergman  \cite{Bergman2014}. 
Mill\'an's  work was particularly helpful in understanding certain properties of $\omega$-maximal ideals.

\begin{proposition}
\label{particiones2}
Let $\{\ideal_k\,\colon\,k\in F\}$ be a collection of maximal ideals on $\N$ with $|F|\leq \omega$. Then, $\{\ideal_k\,\colon\,k\in F\}$ is discrete if and only if there is a partition $\{A_k\,\colon\,k\in F\}$ of $\N$ such that $A_k\in \ideal^*_k\cap(\bigcap_{j\in F\setminus\{k\}}\mathcal I_j)$ for all $k\in F$. 
\end{proposition}

\begin{proof} We assume that $F=\mathbb N$, the other cases are analogous. 
Suppose there is a partition $\{A_k\,\colon\,k\in\mathbb N\}$ of $\N$ such that $A_k\in \ideal^*_k\cap(\bigcap_{j\in\mathbb N\setminus\{k\}}\mathcal I_j)$ for all $k\in\mathbb N$. Then $\{\ideal_j^*\,\colon\,j\in\mathbb N\}\cap A_k^*=\{\mathcal I_k^*\}$ for all $k\in\mathbb N$, thus $\{\ideal_k\,\colon\,k\in\mathbb N\}$ is discrete.

 Now suppose that $\{\ideal_{j}\,\colon\,j\in\mathbb N\}$ is discrete. For each $k\in\mathbb N$, there exists $B_k\subseteq\mathbb N$ such that $\{\ideal_{j}^*\,\colon\,j\in\mathbb N\}\cap B_k^*=\{\mathcal I_{k}^*\}$, that is, $B_k\in\mathcal I_{k}^*\cap\mathcal I_{j}$ for all $j\in\mathbb N$ with $j\neq k$. Set $A_1=B_1$ and $A_j=B_j\setminus\bigcup_{1\leq i<j}B_i$ for all $j\geq2$. Observe that if $k\in\mathbb N$, then $A_k\in\mathcal I_{j}$ for each  $j\in\mathbb N$ with $j\neq k$. We claim that $A_k\in\mathcal I_{k}^*$ for all $k\in\mathbb N$. Indeed, if $k\in\mathbb N$ is given, we have $B_k=A_k\cup(\bigcup_{1\leq i<k}(B_i\cap B_k))$. Notice that $B_i\cap B_k\in\mathcal I_{k}$ for all $1\leq i<k$. Thus, $\bigcup_{1\leq i<k}(B_i\cap B_k)\in\mathcal I_{k}$, and therefore $A_k\in\mathcal I_{k}^*$.

Finally, since $E=\mathbb N\setminus\bigcup_{j\in\mathbb N} A_j\in\mathcal I_{k}$ for all $k\in\mathbb N$,   by substituting $B_1$ by $A_1=B_1\cup E$, we obtain that $\{A_k\,\colon\,k\in\mathbb N\}$ is 
a partition of $\N$ such that $A_k\in \ideal^*_{k}\cap(\bigcap_{j\in\mathbb N\setminus\{k\}}\mathcal I_{j})$ for each $k\in\mathbb N$.
\end{proof}

The following observation is  crucial for  what follows.  Part (1)  turned out to be known \cite[Theorem 3.20]{hindman-strauss} but we include a proof for the sake of completeness. 

\begin{lemma}
\label{sobre puntos limites 1}
Let $\mathfrak D$ be a collection of maximal ideals and  $\mathcal I$ be  a maximal ideal. 

\begin{enumerate}
\item $\mathcal I^*\in\overline{\mathfrak D^*}$ if and only if $\bigcap\mathfrak D\subseteq\mathcal I$. 

\item  Suppose $\mathcal I \in \mathfrak D$. Then, $\mathcal I$ is a limit point of $\mathfrak{D}$ if and only if $\bigcap\mathfrak{D}=\bigcap(\mathfrak{D}\setminus\{\mathcal I\})$. Consequently, $\mathfrak{D}$ is discrete if and only if  $\bigcap(\mathfrak{D}\setminus\{\mathcal K\})\not\subseteq\mathcal K$ for each $\mathcal K\in\mathfrak{D}$.
\end{enumerate}
\end{lemma}

\begin{proof}
\begin{enumerate}
\item $\mathcal I^*\not\in\overline{\mathfrak D^*}$ if and only if there is $A\subseteq \N$ such that $A\in \ideal^*$ and $A^*\cap \mathfrak D^*=\emptyset$ if and only if
there is $A\subseteq \N$ such that $A^c\not\in \ideal^*$ and $A^c\in \idealj^*$ for all $ \idealj^*\in \mathfrak D^*$ if and only if $\bigcap\mathfrak D\not\subseteq\mathcal I$.

\item  It follows from (1).\qedhere
\end{enumerate} 
\end{proof}

It follows that  a $k$-maximal ideal is not $\omega$-maximal, because if $\bigcap_{i=1}^k \ideal_m\subseteq \mathcal J$ for $\mathcal J$ a maximal ideal, then $\mathcal J=\ideal_m$ for some $m$. 
A result already known by Millan \cite{millan}.

\medskip

For each infinite collection of maximal ideals  $\mathfrak{E}=\{\mathcal I_n\,\colon\,n\in\mathbb N\}$ we define a family of infinite subsets of $\N$ which will be very helpful for what follows:
\begin{gather*}
    \mathcal{DI}(\mathfrak{E})=\{M\in[\mathbb N]^\omega\,\colon\,\mbox{$\{\mathcal I_n\,\colon\,n\in M\}$ is discrete}\}.
\end{gather*}
For each $M\subseteq \N$, we  let  $\mathcal I_M=\bigcap_{n\in M}\mathcal I_n$.

\begin{proposition}
\label{strongmax1}
Let $\mathfrak{E}=\{\mathcal I_n\,\colon\,n\in\mathbb N\}$ be a collection of maximal ideals on $\N$ and $A,M\subseteq \N$. Then 

\begin{enumerate}
    \item  $\{\mathcal I_n^*\,\colon\,n\in A\}\subseteq\overline{\{\mathcal I_m^*\,\colon\,m\in M\}}$ if and only if $\mathcal I_M=\mathcal I_{M\cup A}$. 
    
    \item Let $M\in\mathcal{DI}(\mathfrak{E})$ and $n\not\in M$.  Then, $\mathcal I_M\not\subseteq\mathcal I_n$ if and only if $M\cup\{n\}\in\mathcal{DI}(\mathfrak{E})$. 
    
\item Let $M\in\mathcal{DI}(\mathfrak{E})$. Then, $M$ is maximal in $\mathcal{DI}(\mathfrak{E})$  if and only if $\mathcal I_{\mathbb N}=\mathcal I_M$. 
    
\item Let $M\subseteq\mathbb N$ and $D(M)=\{m\in M\,\colon\,\bigcap_{n\in M\setminus\{m\}}\mathcal I_n\not\subseteq\mathcal I_m\}$. Then, $\{\mathcal I_m:m\in D(M)\}$ is discrete. Thus, if $D(M)$ is infinite, then  
$D(M)\in\mathcal{DI}(\mathfrak{E})$.
\end{enumerate}
\end{proposition}

\proof
(1) By Lemma  \ref{sobre puntos limites 1} we have that 
\begin{gather*}
    \mathcal I_M=\mathcal I_{M\cup A}\Longleftrightarrow\mathcal I_M\subseteq\mathcal I_A\Longleftrightarrow\mathcal I_n^*\in\overline{\{{\mathcal I}_m^*\,\colon\,m\in M\}}\,\,\mbox{for all $n\in A$.}
\end{gather*}

(2) By Lemma \ref{sobre puntos limites 1} we have
\begin{gather*}
  \mathcal I_M\not\subseteq\mathcal I_n \Longleftrightarrow  \mathcal I_n^*\not\in\overline{\{\mathcal I_m^*\,\colon\,m\in M\}} \Longleftrightarrow\mbox{$\{\mathcal I_m\,\colon\,m\in M\}\cup\{\mathcal I_n\}$ is discrete}\Longleftrightarrow M\cup\{n\}\in\mathcal{DI}(\mathfrak{E}).
\end{gather*}

(3) Suppose $M$ is maximal in $\mathcal{DI}(\mathfrak{E})$. It suffices to show that $\mathcal I_M\subseteq\mathcal I_\mathbb N$, that is, $\mathcal I_M\subseteq\mathcal I_n$ for all $n\in\mathbb N$. 
Let $n\in\mathbb N$ be given which clearly can be assumed not in $M$.  By the maximality of $M$, $M\cup\{n\}\not\in\mathcal{DI}(\mathfrak{E})$. Thus,  by (2), $\mathcal I_M\subseteq\mathcal I_n$. 

Conversely, suppose that $M$ is not maximal. Thus, there is $n\in\mathbb N\setminus M$ such that
$M\cup\{n\}\in\mathcal{DI}(\mathfrak{E})$. By (2) we have $\mathcal I_M\not\subseteq\mathcal I_n$. Thus  $\mathcal I_{\mathbb N}\neq \mathcal I_M$.

(4) Notice that for all $m\in D(M)$ we have $\bigcap_{n\in D(M)\setminus\{m\}}\mathcal I_n\not\subseteq\mathcal I_m$. By Lemma \ref{sobre puntos limites 1} we conclude that $\{\mathcal I_m:m\in D(M)\}$ is discrete.
\endproof

\begin{theorem}
\label{strongmax2}
Let $\mathfrak{E}= \{\ideal_n:n\in \N\}$ be a collection of maximal ideals on $\N$ and $\ideal=\bigcap_n \ideal_n$.  Then,  $\ideal$ is strongly $\omega$-maximal if and only if  $\mathcal{DI}(\mathfrak{E})$  has a $\subseteq$-maximal element. Moreover, $\mathcal{DI}(\mathfrak{E})$ has at most one maximal element. In particular, a strongly $\omega$-maximal ideal admits a unique representation by a discrete collection of maximal ideals.
\end{theorem}

\begin{proof}
If $M$ is a maximal element of $\mathcal{DI}(\mathfrak{E})$, by Proposition \ref{strongmax1}, $\ideal=\ideal_M$, thus $\ideal$ is strongly $\omega$-maximal. 

Conversely, suppose  $\ideal$ is strongly $\omega$-maximal and let $\mathfrak{D}=\{\mathcal{K}_n:n\in \N\}$ be  a discrete collection  of maximal ideals on $\N$ such that $\ideal=\bigcap_n \mathcal{K}_n$. By Proposition \ref{particiones2}, there is a partition  $\{A_n:n\in \N\}$ of $\N$ such that $A_n\not\in \mathcal{K}_n$ for all $n\in \N$.  We will show that for each $n$ there is a unique $l_n$ such that $\mathcal I_{l_n}=\mathcal K_n$. Then,  letting  $M=\{l_n\,\colon\,n\in \N\}$ we have that $M\in \mathcal{DI}(\mathfrak{E})$, as $\mathfrak{D}$ is discrete, $\mathcal I=\mathcal I_M$ and $M$ is maximal by Proposition \ref{strongmax1}.

Given $n\in \N$, as $A_n\not\in \ideal$ for all $n$, there is $m$ such that $A_n\not\in \ideal_m$. We claim that $\ideal_m\restriction A_n= \mathcal{K}_n\restriction A_n$. Indeed, as the $A_n$'s are pairwise disjoint, $\ideal\restriction A_n= \mathcal{K}_n\restriction A_n$. In particular, $\mathcal{K}_n\restriction A_n\subseteq \ideal_m\restriction A_n$, but $\ideal_m\restriction A_n$ and $\mathcal{K}_m\restriction A_n$ are maximal ideals on $A_n$, thus
$\ideal_m\restriction A_n= \mathcal{K}_n\restriction A_n$. Therefore $\ideal_m= \mathcal{K}_n$. Notice that this argument shows that for all $n$ there is a unique $l_n$ such that $\ideal_{l_n}= \mathcal{K}_n$.

Finally,  suppose $M_1$ and $M_2$ are two maximal elements of $\mathcal{DI}(\mathfrak{E})$. By Proposition \ref{strongmax1}, $\ideal=\ideal_{M_1}=\ideal_{M_2}$. Let $\{A_n:n\in M_1\}$ be a partition   of $\N$ such that $A_n\not\in \mathcal{I}_n$ for all $n\in M_1$.  Let $n_1\in M_1\setminus M_2$. We have shown above that $n_1$ is the unique $m\in \N$ such that $A_{n_1}\not\in \ideal_m$. Thus $A_{n_1}\in \ideal_m$ for all $m\in M_2$. Thus $A_{n_1}\in \ideal_{M_2}=\ideal$, which contradicts that $A_{n_1}\not\in \ideal_{n_1}$.

For the last claim, suppose $\mathfrak{E}_1$  and $\mathfrak{E}_2 $ are two discrete countable collections of maximal ideals such that $\bigcap \mathfrak{E}_1=\bigcap \mathfrak{E}_2$. Let $\{J_n: n\in \N\}$ be an enumeration (without repetitions) of $\mathfrak{E}_1\cup \mathfrak{E}_2$. 
Let  $L_i=\{n\in \N: J_n\in \mathfrak{E}_i\}$ and notice that  $L_i$ is a maximal element of $\mathcal{DI}(\mathfrak{E}_1\cup \mathfrak{E}_2)$ by 
Proposition \ref{strongmax1}. Since $\bigcap \mathfrak{E}_1=\bigcap (\mathfrak{E}_1\cup \mathfrak{E}_2)$,  $\mathcal{DI}(\mathfrak{E}_1\cup \mathfrak{E}_2)$ has a unique maximal element.  Thus $L_1=L_2$ and $\mathfrak{E}_1=\mathfrak{E}_2$.
\end{proof}

We recall that a topological space $X$ is {\em scattered} if every non-empty subspace of $X$ has an isolated point.   
We say that a collection $\mathfrak{E}$ of maximal ideals on $\N$ is {\em scattered} if $\mathfrak{E}^*$ 
is a scattered subspace of $\beta\N$.

\begin{theorem}
\label{scattered}
Let $\ideal$ be an ideal on $\N$. Then $\ideal=\bigcap \mathfrak{E}$ for an infinite countable scattered collection $\mathfrak{E}$ of maximal ideals on $\N$ if and only if $\ideal$ is strongly $\omega$-maximal. 
\end{theorem}

\begin{proof}
Suppose $\mathfrak{E}^*$ is scattered.  Since the collection of isolated points of $\mathfrak{E}^*$ is discrete and dense in $\mathfrak{E}^*$ (see \cite[p. 150]{semadeni}), there is  $M\in \mathcal{DI}(\mathfrak{E})$ such that $\{\ideal^*:\,n \in M\}$ is dense in $\mathfrak{E}^*$, then easily $M$ is maximal in $\mathcal{DI}(\mathfrak{E})$. Thus by Theorem \ref{strongmax2}, $\ideal$ is strongly $\omega$-maximal. 
The converse obviously follows from Theorem \ref{strongmax2}. 
\end{proof}

\begin{remark}
\label{minimal-repre}
Following Kadets, Seliutin and Tryba \cite{kadetsetal2022}, a collection  of non principal ultrafilters $\mathfrak{M}$ is said {\em minimal}, if $\bigcap \mathfrak{M}\neq\bigcap (\mathfrak{M}\setminus\{\mathcal{F}\})$ for every $\mathcal{F}\in \mathfrak{M}$, and a filter $\mathcal{F}$ has a {\em minimal representation} if $\mathcal F=\bigcap \mathfrak M$ for some  minimal collection $\mathfrak{M}$ of ultrafilters.   Lemma \ref{sobre puntos limites 1} says that $\mathfrak{M}$ is minimal if and only if it is a discrete subset of  $\beta\N$. And, from Theorem \ref{scattered} it follows that a filter $\mathcal{F}$ has a countable infinite minimal representation if and only if $\mathcal{F}^*$ is strongly $\omega$-maximal ideal.
They also showed that  whenever a filter admits a countable minimal representation,  it is unique. This also follows from Theorem \ref{strongmax2}. 
\end{remark}

We present two examples. The first one shows that  a strongly $\omega$-maximal ideal can also be represented by a  non-discrete countable collection of maximal ideals. The second one provides an example of a $\omega$-maximal ideal which is not strongly $\omega$-maximal.

\begin{example}
Let $\mathfrak{E}= \{\ideal_n:n\in \N\}$ be a discrete infinite collection of maximal ideals on $\N$ and $\ideal=\bigcap_n \ideal_n$. Let $\{A_n:n\in \N\}$ be a partition of $\N$ such that $A_n\not\in \ideal_n$ for all $n\in\N$. Let $\idealj_0$ be the ideal generated by $\ideal\cup \{A_n:n\in \N\}$. Then $\idealj_0$ is non-trivial, in fact, if $\N=B\cup A_0\cup\ldots \cup A_m$, then $A_{m+1}\subseteq B$, thus $B\not\in \ideal$ and hence $\N\not\in \idealj_0$.  Let $\idealj$ be a maximal ideal extending $\idealj_0$. Then $\ideal=\bigcap ( \mathfrak{E}\cup \{\idealj\})$ but  $\mathfrak{E}\cup \{\idealj\}$ is not discrete.  
\end{example}

\begin{example}
\label{Ejemplo-no-omega-max}
Let $\mathcal{I}$ be an ideal on $\N$  (containing all finite sets). Define a topology $\tau_\mathcal{I}$ over $\seq$ by letting a subset $U$ of $\seq$ be open, if $\{n\in \N: \; s\widehat{\;}n \not\in U\}\in\mathcal{I}$ for all $s\in U$.
Then $(\seq, \tau_\mathcal{I})$ is Hausdorff, zero-dimensional  and without isolated points. Moreover,  when  $\mathcal{I}$ is a  maximal ideal, $(\seq, \tau_\mathcal{I})$ is extremally disconnected, i.e., the closure of an open set is open (see \cite{Louveau72}  and \cite{Dow1988}). Since this space is zero-dimensional and has no isolated points, for every discrete $D\subseteq \seq$ there is $s\in \seq\setminus D$ such that $D\cup\{s\}$ is still discrete, i.e., there are no maximal discrete subsets of $\seq$.   

It is a classical fact that every countable extremally disconnected Hausdorff space can be embedded in $\beta\N$ (see, for instance, \cite[Theorem 1.4.7]{vanMillHandbook}). Let $\mathfrak{X}\subseteq \beta\N$ be  homeomorphic to $(\seq, \tau_\mathcal{I})$. Thus,  $\mathfrak{X}$ has no maximal discrete subsets. As usual, $\mathcal{U}^*$  denotes the dual ideal of a filter $\mathcal{U}$. Let $\mathfrak{E}=\{\mathcal{U}^*:\mathcal{U}\in \mathfrak X\}$.
Then  $\mathcal{J}=\bigcap \mathfrak{E}$ is an $\omega$-maximal ideal which is not strongly $\omega$-maximal (by Theorem \ref{strongmax2}). 

\end{example}

Our next result implies that the $\omega$-maximal ideal constructed in the previous example is nowhere maximal. 

\begin{proposition}
    Let $\mathfrak{D}=\{\mathcal I_n\,\colon\,n\in\mathbb N\}$ be a collection of maximal ideals and $\mathcal I=\bigcap_{n\in\mathbb N}\mathcal I_n$. Then, $\mathfrak{D}$ has an isolated point if and only if there is
    $A\in\mathcal I^+$ such that $\mathcal I\restriction A$ is a maximal ideal on $A$.
\end{proposition}

\begin{proof}
We assume that the enumeration of $\mathfrak{D}$ is without repetitions. Suppose that $\mathcal I_n^*$ is an isolated point of $\mathfrak{D}^*$, that is, there is $A\subseteq\mathbb N$ such that $\mathfrak{D}^*\cap A^*=\{\mathcal I^*_n\}$. Then, $A\in\mathcal I^+$ and $\mathcal I\restriction A=\mathcal I_n\restriction A$ is a maximal ideal on $A$. Conversely, suppose that there is $A\in\mathcal I^+$ such that $\mathcal I\restriction A$ is a maximal ideal on $A$. Notice that $\mathcal I\restriction A=\mathcal I_m\restriction A$ for some unique $m\in\mathbb N$. Thus, $\mathcal I_m^*$ is an isolated point of $\mathfrak D^*$.
\end{proof}

\section{Dimension of $c_{0,\mathcal J}/c_{0,\mathcal I}$}
\label{dimension-quotient}

In this section we  present several results about the dimension of the quotient $c_{0,\mathcal J}/c_{0,\mathcal I}$ when $\mathcal I$ and $\mathcal J$ are ideals on $\N$ with $\mathcal I\subsetneq\mathcal J$. In particular, we obtain some information about $\ell_\infty/c_{0,\mathcal I}$ which corresponds to the case $\idealj=\mathcal{P}(\N)$.

Now we introduce a relativized version of $k$-maximality.

\begin{definition}
\label{k-maximal def}
Let $\mathcal I$ and $\mathcal J$ be ideals on $\mathbb N$ with $\mathcal I\subsetneq\mathcal J$ and $k$ be a positive integer. We say that $\mathcal I$ is $k$-maximal in $\mathcal J$  if there is a family $\{\mathcal L_1,\ldots,\mathcal L_k\}$ of pairwise distinct maximal ideals on $\mathbb N$ with  $\mathcal J\not\subseteq\mathcal L_i$ for all $1\leq i\leq k$ and $\mathcal I=(\bigcap_{i=1}^k\mathcal L_i)\cap\mathcal J$.
\end{definition}

It is clear that, for a positive integer $k$,  an ideal is $k$-maximal in the trivial ideal $\mathcal{P}(\N)$ if and only if it is $k$-maximal as it was  defined in Section \ref{kappa-max}. To provide  a characterization of the relativized version of $k$-maximality we need to  introduce a special type of disjoint families.

\begin{definition}\label{I-J-k family}
Let $\mathcal I$ and $\mathcal J$ be ideals on $\mathbb N$ with $\mathcal I\subsetneq\mathcal J$ and $k$ be a positive integer.
A collection $\{A_1,\ldots,A_k\}$  of subsets of $\N$ is called a $(\mathcal I,\mathcal J,k)$-family if 

\begin{enumerate}
\item $\{A_1,\ldots,A_k\}\subseteq\mathcal J\setminus\mathcal I$;

\item  $A_i\cap A_j=\emptyset$ if $i\neq j$;

\item $B\setminus(A_1\cup\cdots\cup A_k)\in\mathcal I$ for all $B\in\mathcal J\setminus\mathcal I$;

\item $\mathcal I\restriction A_i$ is a maximal ideal on $A_i$ for each $i\in\{1,\ldots,k\}$.
\end{enumerate}

\end{definition}

\begin{proposition}
\label{Carac-k-maximal-J}
Let $\mathcal I$ and $\mathcal J$ be ideals on $\mathbb N$ with $\mathcal I\subsetneq\mathcal J$  and $k$ a positive integer. Then $\mathcal I$ is $k$-maximal in $\mathcal J$ if and only if  there exists a $(\mathcal I,\mathcal J,k)$-family. Moreover, $\mathcal I$ is $k$-maximal if and only if there is a partition $\{A_1,\ldots,A_k\}$  of $\N$ such that $\ideal\restriction A_i$ is a maximal ideal on $A_i$.
\end{proposition}

\begin{proof}
Suppose there is  a $(\mathcal I,\mathcal J,k)$-family $\{A_1,\ldots,A_k\}$. Define $\mathcal L_i=(\mathcal I\restriction A_i)\sqcup\mathcal P(A_i^c)$ for each $i=1,\ldots,k$. By Lemma \ref{condicion equivalente a maximalidad} each $\mathcal L_i$ is maximal on $\mathbb N$ and $A_i\in\mathcal J\setminus\mathcal L_i$ for all $1\leq i\leq k$. Since $A_i\in\mathcal L_j$ for all $j\neq i$, we have $\mathcal L_i\neq\mathcal L_j$ for $i\neq j$. We claim that $\mathcal I=(\bigcap_{i=1}^k\mathcal L_i)\cap\mathcal J$. The inclusion $\subseteq$ is clear. Now, suppose $B\in(\bigcap_{i=1}^k\mathcal L_i)\cap\mathcal J$. Then $\bigcup_{i=1}^k(B\cap A_i)\in\mathcal I$.  As $B=B\setminus(A_1\cup\cdots \cup A_k) \cup (B\cap (A_1\cup\cdots \cup A_k))$. By  condition (3),  $B\in\mathcal I$.

Conversely, suppose  $\ideal$ is $k$-maximal in $\idealj$ and let  $\{\mathcal L_1,\ldots,\mathcal L_k\}$ be a family of maximal ideals on $\mathbb N$ with $\mathcal J\not\subset\mathcal L_i$ for all $1\leq i\leq k$ and $\mathcal I=(\bigcap_{i=1}^k\mathcal L_i)\cap\mathcal J$.
By Proposition \ref{particiones2}, there is a partition $\{D_1,\ldots,D_k\}$ of $\N$  such that  $D_i\in \mathcal{L}^*_i\cap\mathcal L_j$ for all $j,  i$ in $\{1,\ldots, k\}$ with $j\neq i$. For each $1\leq i\leq k$, let $X_i\in\mathcal J\setminus\mathcal L_i$. Set $A_i=X_i\cap D_i$ for $i\in\{1,\ldots,k\}$. Then  $A_i\in\mathcal J\setminus\mathcal L_i$ for all $i$.
We will prove that $\{A_1,\ldots,A_k\}$ is a $(\mathcal I,\mathcal J,k)$-family. Clearly, the conditions (1) and (2)  of Definition \ref{I-J-k family} hold. It remains to check conditions (3) and (4):
      
\begin{enumerate}
\item[(3)] Let $B\in\mathcal J\setminus\mathcal I$. Observe that $B\setminus A_i\in\mathcal L_i$ for all $i\in\{1,\ldots,k\}$. So, $B\setminus(A_1\cup\ldots\cup A_k)\in\mathcal I$.
     
\item[(4)] Fix $i\in\{1,\ldots,k\}$. Then $\mathcal I\restriction A_i=\mathcal L_i\restriction A_i$. Hence, $\mathcal I\restriction A_i$ is maximal on $A_i$.
\end{enumerate}

For the last statement, observe that an ideal $\mathcal I$ is $k$-maximal in $\mathcal P(\mathbb N)$ if and only if $\ideal$ is $k$-maximal.
\end{proof}

\begin{remark}$(\mathcal I,\mathcal J,k)$-families are unique in the following sense. Let $\mathcal I$ and $\mathcal J$ be ideals on $\mathbb N$ with $\mathcal I\subsetneq\mathcal J$.
Suppose that $\mathcal I$ is $k$-maximal in $\mathcal J$ and $\{A_1,\ldots,A_k\}$ is a  $(\mathcal I,\mathcal J,k)$-family. Then 

\begin{enumerate}
\item[(i)]  $\mathcal J=\mathcal I\sqcup\mathcal P(A_1\cup\ldots\cup A_k)$.

\item[(ii)] If $\{A_1',\ldots,A_k'\}$ is also a $(\mathcal I,\mathcal J,k)$-family, then $(\bigcup_{1\leq i\leq k} A_i)\triangle(\bigcup_{1\leq i\leq k} A_i')\in\mathcal I$. 
\end{enumerate}
\end{remark}

\begin{lemma}
\label{linear independence}
Let $\mathcal I$ and $\mathcal J$ be ideals on $\mathbb N$ with $\mathcal I\subseteq \mathcal J$. If $\mathcal A\subseteq \mathcal J$ is a $\mathcal I$-AD family, then the set $\{\chi_A+c_{0,\mathcal I}\,\colon\,A\in\mathcal A\}$ is a linearly independent subset of $c_{0,\mathcal J}/c_{0,\mathcal I}$.
\end{lemma}

\begin{proof}
Let $A_1,\ldots,A_m\in\mathcal A$ be such that $\sum_{j=1}^ma_j(\chi_{A_j}+c_{0,\mathcal I})={\bf0}$ for some $\{a_j\,\colon\,1\leq j\leq m\}\subset\mathbb K$. Let $D_1=A_1$ and $D_j=A_j\setminus(A_1\cup\cdots\cup A_{j-1})$ for $2\leq j\leq m$. Since $A_i\cap A_j\in\mathcal I$ for all $i,j\in\{1,\ldots,k\}$ with $i\neq j$,  we have $D_j\in\mathcal J\cap \mathcal I^+$  and $\chi_{A_j}+c_{0,\mathcal I}=\chi_{D_j}+c_{0,\mathcal I}$ for all $j\in\{1,\ldots,m\}$. So, $\sum_{j=1}^ma_j(\chi_{D_j}+c_{0,\mathcal I})={\bf0}$. Thus, $a_j=0$ for all $1\leq j\leq m$.
\end{proof}

\begin{theorem}
\label{I+particion}
Let $\mathcal I$ and $\mathcal J$ be ideals on $\mathbb N$ with $\mathcal I\subsetneq \mathcal J$ and $k$ be a positive integer. Then $\dim (c_{0,\mathcal J}/c_{0,\mathcal I})=k$ if and only if $\mathcal I$ is $k$-maximal in $\mathcal J$.
\end{theorem}

\begin{proof}
Suppose that $\dim (c_{0,\mathcal J}/c_{0,\mathcal I})=k$. By \cite[Corollary 1, p. 70]{schaeffer}, there is an order isomorphism $\phi\colon c_{0,\mathcal J}/c_{0,\mathcal I}\to\mathbb R^k$. Let $\Lambda\colon c_{0,\mathcal J}\to\mathbb R^k$ be given by $\Lambda({\bf y})=\phi({\bf y}+c_{0,\mathcal I})$ if ${\bf y}\in c_{0,\mathcal J}$. Notice that $\Lambda$ is an onto Banach lattice homomorphism. For each $j\in\{1,\ldots,k\}$, let ${\bf y}_j\in c_{0,\mathcal J}\setminus c_{0,\mathcal I}$ be such that $\Lambda({\bf y}_j)=e_j$. Since $\Lambda$ is a lattice homomorphism, we have $\Lambda(|{\bf y}_j|)=|\Lambda({\bf y}_j)|=|e_j|=e_j$. Fix $j\in\{1,\ldots,k\}$ and let $\varepsilon_j>0$ be such that $A_j:=A(\varepsilon_j,{\bf y}_j)\not\in\mathcal I$. So, $\varepsilon_j\chi_{A_j}\leq|{\bf y}_j|$. By the order preserving-ness of $\Lambda$, we conclude that $\Lambda(\chi_{A_j})=a_je_j$ for some $a_j>0$. Notice that if $i,j\in\{1,\ldots,k\}$ and $i\neq j$, we have $A_i\cap A_j\in\mathcal I$ because of $\Lambda(\chi_{A_i\cap A_j})=\Lambda(\chi_{A_i}\wedge\chi_{A_j})=a_ie_i\wedge a_je_j={\bf0}$. Thus,
$\{A_1,\ldots,A_k\}$ is a $\mathcal I$-AD family.  As in the proof of Lemma \ref{linear independence}, we will assume that the family is pairwise disjoint.
From Lemma \ref{linear independence} we obtain that $\{\chi_{A_j}+c_{0,\mathcal I}\,\colon\,1\leq j\leq k\}$ is a basis for $c_{0,\mathcal J}/c_{0,\mathcal I}$. We claim that $\{A_1,\ldots,A_k\}$ is a $(\mathcal I,\mathcal J,k)$-family:

\begin{enumerate}
\item Clearly, $A_i\in\mathcal J\setminus\mathcal I$ for all $i\in\{1,\ldots,k\}$.

\item $\{A_j: 1\leq j\leq k\}$ is pairwise disjoint by construction.

\item Let $B\in\mathcal J\setminus\mathcal I$ be given and $D\coloneqq B\setminus(A_1\cup\cdots\cup A_k)$. Then there are $\alpha_1,\ldots,\alpha_k\in\mathbb R$ such that ${\bf z}=\chi_D-\sum_{1\leq j\leq k}\alpha_j\chi_{A_j}\in c_{0,\mathcal I}$. Whence, $D\subseteq A(1/2,{\bf z})\in\mathcal I$.
    
\item Fix $1\leq i\leq k$ and let $B\subseteq A_i$.
Since $B\in\mathcal J$, there exist $\beta_1,\ldots,\beta_k\in\mathbb R$ such that ${\bf w}=\chi_{B}-\sum_{1\leq j\leq k}\beta_j\chi_{A_j}\in c_{0,\mathcal I}$. If $\beta_i=0$, then $B\subseteq A(1/2,{\bf w})\in\mathcal I$. If $\beta_i\neq0$, then $A_i\setminus B\subseteq A(|\beta_i|,{\bf w})\in\mathcal I$. Thus, $\mathcal I\restriction A_i$ is maximal on $A_i$.
\end{enumerate}

Conversely, suppose that $\mathcal I$ is $k$-maximal in $\mathcal J$ and let $\mathcal L_1,\ldots,\mathcal L_k$ be maximal ideals on $\mathbb N$ such that $\mathcal I=(\bigcap_{1\leq i\leq k}\mathcal L_j)\cap\mathcal J$. Define $\Phi\colon c_{0,\mathcal J}\to\mathbb R^k$ by $\Phi({\bf x})=(\mathcal L_1^*-\lim{\bf x},\ldots,\mathcal L_k^*-\lim{\bf x})$ if ${\bf x}\in c_{0,\mathcal J}$. Notice that $\Phi$ is linear and continuous.

We claim that $\ker\Phi=c_{0,\mathcal I}$. If $\Phi({\bf x})={\bf0}$, then $\mathcal L_j^*-\lim{\bf x}=0$ for all $1\leq j\leq k$. Thus, $A(\varepsilon,{\bf x})\in (\bigcap_{1\leq i\leq k}\mathcal L_j)\cap\mathcal J=\mathcal I$ for each $\varepsilon>0$. Whence, $\ker\Phi\subset c_{0,\mathcal I}$. The converse is clear. 

Finally, we show that $\Phi$ is onto. Indeed, there are $A_1,\ldots,A_k\in\mathcal J$ satisfying $A_i\in\mathcal L_i^+\cap(\bigcap_{1\leq j\neq i\leq k}\mathcal L_j)$ for all $i\in\{1,\ldots,k\}$. So, $\Phi(\chi_{A_i})=e_i$ for each $1\leq i\leq k$. Therefore, $\Phi$ is onto and we conclude that $\dim c_{0,\mathcal J}/c_{0,\mathcal I}=k$.
\end{proof}

From Theorem \ref{I+particion} we obtain the following result.

\begin{corollary}
\label{dim of l/c_I and k-maximal}
Let $\mathcal I$ be a proper ideal on a set $\mathbb N$, and $k$ be a positive integer. Then $\dim(\ell_\infty/c_{0,\mathcal I})=k$ if and only if $\mathcal I$ is  the intersection of exactly $k$ maximal ideals on $\mathbb N$.  
\end{corollary}

\section{Structural properties of complemented ideals}
\label{complementacion}

 In this section, we study the structural properties of complemented ideals $\ideal$, that is, ideals for which the Banach space $c_{0,\mathcal I}$ is complemented in $\ell_\infty$. We recall that a classical theorem  of Lindenstrauss \cite{lindenstrauss} states that an infinite dimensional closed subspace $X$ of $\ell_\infty$ is complemented if and only if $X$ is isomorphic to $\ell_\infty$.
As noted by Kania in \cite{kania},  the intersection of finitely many maximal ideals is complemented (for a proof, see  \cite[Proposition 5.20]{rincon-uzcategui}). Therefore, if $\mathcal I$ is such an ideal, then $c_{0,\mathcal I}$ is isomorphic to $\ell_\infty$. Recently, Hru\v{s}\'ak and S\'aenz characterized the complementation of $c_{0,\mathcal I}$ in terms of the quotient $\ell_\infty/c_{0,\mathcal I}$ \cite{hrusak-saenz2025}. Before stating their result, we need to introduce some notation.

 We say that a proper ideal $\mathcal I$ on a set $F\subseteq\mathbb N$ is  {\em complemented in $F$} if $c_{0,\mathcal I}$ is complemented in $\ell_\infty(F)$ (recall that $\ell_\infty(F)$ and $c_{0,\mathcal I}$ consist of sequences of the form $(x_n)_{n\in F}$).

For an ideal $\mathcal I$, let $K_{\mathcal I}=\{p\in\beta\mathbb N\,\colon\,\mathcal I^*\subseteq p\}$. It is known that $K_\mathcal I$ is a compact Hausdorff space, since it is a closed subset of $\beta\mathbb N$.  Moreover, the map 
\begin{align*}
    \Psi\colon\ell_\infty/c_{0,\mathcal I}&\to C(K_\mathcal I)\\
    {\bf x}+c_{0,\mathcal I}&\mapsto\Psi({\bf x}+c_{0\mathcal I})\,\colon\, p\in K_\mathcal I\mapsto p-\lim{\bf x}
\end{align*}
is a Banach lattice isometry. Therefore, we have the following:

\begin{theorem}
\label{kania}(Kania \cite{kania})
$\ell_\infty/c_{0,\mathcal I}$ is Banach lattice isometric to $ C(K_\mathcal I)$.
\end{theorem}

\begin{theorem} (Hru\v{s}\'ak and S\'aenz \cite{hrusak-saenz2025}).
\label{teorema michael-saens}
    Let $\mathcal I$ be an ideal on $\mathbb N$. Then, the following statements are equivalent:
    \begin{enumerate}
        \item $\mathcal I$ is complemented.
        \item $\ell_\infty/c_{0,\mathcal I}$ is isomorphic to a closed subspace of $\ell_\infty$. 
        \item $\ell_\infty/c_{0,\mathcal I}$ is isometric to a closed subspace of $\ell_\infty$. 
        \item $C(K_\mathcal I)$ is isomorphic to a closed subspace of $\ell_\infty$. 
    \end{enumerate}
\end{theorem}

\begin{corollary}\label{complementacion es equivalente a operador con kernel igual al ideal}
    Let $\mathcal I$ be an ideal on $\mathbb N$. Then, $\mathcal I$ is complemented if and only if $c_{0,\mathcal I}$ is the kernel of a continuous operator from $\ell_\infty$ onto a closed subspace of $\ell_\infty$.
\end{corollary}

\begin{proof}
 If $\mathcal I$ is complemented and $P\colon\ell_\infty\to c_{0,\mathcal I}$ is a projection, then $c_{0,\mathcal I}$ is a kernel of $Q\coloneqq Id-P\colon\ell_\infty\to\ell_\infty$, and notice that $Q(\ell_\infty)=\ker P$ is closed. Conversely, let $Q\colon\ell_\infty\to\ell_\infty$ be a continuous operator such that $Q(\ell_\infty)$ is closed and $\ker Q=c_{0,\mathcal I}$. Then, the map $\hat Q\colon\ell_\infty/c_{0,\mathcal I}\to Q(\ell_\infty)$ defined by ${\bf x}+c_{0,\mathcal I}\mapsto Q({\bf x})$, is an embedding by the open mapping theorem. By Theorem \ref{teorema michael-saens}, $\mathcal I$ is complemented.
\end{proof}

Our first aim is to provide an analogous characterization to Theorem \ref{teorema michael-saens} for the complementation of an ideal on a set $F$. Before doing so, we need some technical tools.

\medskip

Let $\mathcal I$ be an ideal on a countably infinite set $F$. For each sequence ${\bf x}=(x_n)_{n\in F}$  in $\mathbb R$, we set $B_{\bf x}(\mathcal I)=\{b\in\mathbb R\,\colon\,\{k\in F\,\colon\,x_k>b\}\not\in\mathcal I\}$ and 
\begin{gather*}
 \mathcal I-\limsup{\bf x}=
 \begin{cases}
\sup B_{\bf x}(\mathcal I),& B_{\bf x}(\mathcal I)\neq\emptyset;\\
-\infty, & B_{\bf x}(\mathcal I)=\emptyset.
 \end{cases}
 \end{gather*}
 Notice that if ${\bf x}\in\ell_\infty(F)$ and $\mathcal I$ is proper, then $B_{\bf x}(\mathcal I)\neq\emptyset$ and $-\infty<\displaystyle\mathcal I-\limsup{\bf x}=\sup B_{\bf x}(\mathcal I)<\infty$. 

 \begin{remark}\label{norm in the quotient}
     If $\mathcal I$ is an ideal on a countably infinite set $F$, then the proof of \cite[Lemma 5.7]{rincon-uzcategui} shows that
     \begin{equation*}
         \|{\bf x}+c_{0,\mathcal I}\|=\mathcal I-\limsup{\bf x}, \quad {\bf x}\in\ell_\infty(F).
     \end{equation*}
 \end{remark}

\begin{lemma}\label{isometries between linfty and linfty}
  Let $\mathcal I$ be an ideal on $\mathbb N$, $f\colon\mathbb N\to F$ be a bijection, and $\mathcal J=f(\mathcal I)\coloneqq\{f(A)\,\colon\,A\in\mathcal I\}$. The map
   \begin{equation*}
      \tilde f\colon {\bf x}+c_{0,\mathcal I}\in\ell_\infty/c_{0,\mathcal I}\mapsto {\bf x}\circ f^{-1}+c_{0,\mathcal J}\in\ell_\infty(F)/c_{0,\mathcal J}
   \end{equation*}
   is an onto linear isometry.
\end{lemma}

\begin{proof}
We start by proving the following claim.
\begin{claim}\label{previous isometry}
    The map
   \begin{equation*}
       \hat f\colon {\bf x}=(x_n)\in\ell_\infty\mapsto {\bf x}\circ f^{-1}=(x_{f^{-1}(n)})_{n\in F}\in\ell_\infty(F)
   \end{equation*}
   is an onto linear isometry. Moreover, ${\bf x}\in c_{0,\mathcal I}$ if and only if $\hat f({\bf x})={\bf x}\circ f^{-1}\in c_{0,\mathcal J}$, and
    \begin{equation*}
         \mathcal J-\limsup\hat{f}({\bf x})=\mathcal I-\limsup{\bf x},\quad \text{for all }{\bf x}\in\ell_\infty.
     \end{equation*}
\end{claim}

\begin{proof}[Proof of the claim]
     Clearly, $\hat{f}$ is linear. If ${\bf x}=(x_n)\in\ell_\infty$, then
    \begin{equation*}
        \|\hat{f}({\bf x})\|=\sup_{n\in F}|x_{f^{-1}(n)}|=\sup_{m\in\mathbb N}|x_m|=\|{\bf x}\|.
    \end{equation*}
    The ontoness is easy: if ${\bf y}\in\ell_\infty(F)$, then ${\bf y}\circ f\in\ell_\infty$ and $\hat{f}({\bf y}\circ f)={\bf y}.$ 

   To show the second part, notice that for each $\varepsilon>0$, $A(\varepsilon,{\bf x}\circ f^{-1})=f(A(\varepsilon,{\bf x}))$. Thus,
    \begin{align*}
        \hat f({\bf x})={\bf x}\circ f^{-1}\in c_{0,\mathcal J}&\Longleftrightarrow(\forall\varepsilon>0)(A(\varepsilon,{\bf x}\circ f^{-1})\in\mathcal J)\\
        &\Longleftrightarrow(\forall\varepsilon>0)(A(\varepsilon,{\bf x})\in\mathcal I)\\
         &\Longleftrightarrow {\bf x}\in c_{0,\mathcal I}.
    \end{align*}
Finally, if ${\bf x}\in\ell_\infty$, then for each $A\in\mathbb R$, we have
    \begin{align*}
        \mathcal J-\limsup\hat{f}({\bf x})>A&\Longleftrightarrow(\exists b\in\mathbb R)(b>A\text{ and }\{k\in F\,\colon\,x_{f^{-1}(n)}>b\}\not\in\mathcal J)\\
        &\Longleftrightarrow(\exists b\in\mathbb R)(b>A\text{ and }\{m\in\mathbb N\,\colon\,x_m>b\}\not\in\mathcal I)\\
        &\Longleftrightarrow\mathcal I-\limsup{\bf x}>A.\qedhere
    \end{align*}
\end{proof}
   
Now we continue the proof of the lemma.

\medskip

If $q\colon\ell_\infty(F)\to\ell_\infty(F)/c_{0,\mathcal J}$ is the quotient map, then $\ker q\circ\hat f=c_{0,\mathcal I}$ by Claim \ref{previous isometry}. Thus, the induced map
     \begin{equation*}
      \tilde f\colon {\bf x}+c_{0,\mathcal I}\in\ell_\infty/c_{0,\mathcal I}\mapsto \hat{f}({\bf x})+c_{0,\mathcal J}= {\bf x}\circ f^{-1}+c_{0,\mathcal J}\in\ell_\infty(F)/c_{0,\mathcal J}
   \end{equation*}
   is a linear bijection. Finally, by Claim \ref{previous isometry} and Remark \ref{norm in the quotient}, we have for each ${\bf x}\in\ell_\infty$, 
   \begin{equation*}
       \|\tilde f({\bf x}+c_{0,\mathcal I})\|=\|\hat f({\bf x})+c_{0,\mathcal J}\|= \mathcal J-\limsup\hat{f}({\bf x})=\mathcal I-\limsup{\bf x}=\|{\bf x}+c_{0,\mathcal I}\|.\qedhere
   \end{equation*}
\end{proof}

\begin{theorem}\label{complementation of ideals in countable sets}
     Let $\mathcal J$ be an ideal on a countably infinite set $F$. Then, $\mathcal J$ is complemented in $F$ if and only if $\ell_\infty(F)/c_{0,\mathcal J}$ is isometric to a closed subspace of $\ell_\infty(F)$.
\end{theorem}

\begin{proof}
Suppose that $\mathcal J$ is complemented in $F$. Then, $\ell_\infty(F)=c_{0,\mathcal J}\oplus Y$ for some closed subspace $Y\subseteq\ell_\infty(F)$. Let $f\colon\mathbb N\to F$ be a bijection and consider the map $\hat{f}\colon\ell_\infty\to\ell_\infty(F)$ defined as in Claim \ref{previous isometry}. Then,
\begin{equation*}
    \ell_\infty=\hat f^{-1}(\ell_\infty(F))=\hat f^{-1}(c_{0,\mathcal J}\oplus Y)=c_{0,\mathcal I}\oplus W, 
\end{equation*}
 where $W=\hat{f}^{-1}(Y)$ and $\mathcal I=\{f^{-1}(B)\,\colon\,B\in\mathcal J\}$. Thus, $\mathcal I$ is complemented. By Theorem \ref{teorema michael-saens}, $\ell_\infty/c_{0,\mathcal I}$ is isometric to a closed subspace of $\ell_\infty$. Since $\mathcal J=f(\mathcal I)$, from Lemma \ref{isometries between linfty and linfty}, it follows that $\ell_\infty(F)/c_{0,\mathcal J}$ is isometric to a closed subspace of $\ell_\infty(F)$.

   Conversely, suppose that $\ell_\infty(F)/c_{0,\mathcal J}$ is isometric to a closed subspace of $\ell_\infty(F)$.  Then, $\mathcal I=\{f^{-1}(B)\,\colon\,B\in\mathcal J\}$ is an ideal on $\mathbb N$ such that $\mathcal J=f(\mathcal I).$ By Lemma \ref{isometries between linfty and linfty}, $\ell_\infty/c_{0,\mathcal I}$ is isometric to a closed subspace of $\ell_\infty$. Hence, $c_{0,\mathcal I}$ is complemented in $\ell_\infty$ by Theorem \ref{teorema michael-saens}. Again, by Lemma \ref{isometries between linfty and linfty}, we conclude that $c_{0,\mathcal J}$ is complemented in $\ell_\infty(F)$.
\end{proof}

Recall that if $|F|\leq\omega$ and $(X_j)_{j\in F}$ is a sequence of Banach spaces,
the space  $\ell_\infty((X_j)_{j\in F})$ denotes their $\ell_\infty$-sum, i.e., the Banach space of all bounded sequences $\displaystyle(x_j)\in\prod_{j\in F}X_j$, equipped with the norm $\|\cdot\|$ given by $\displaystyle\|(x_j)\|=\sup_{j\in F}\|x_j\|$. 

\begin{remark}\label{l_infty-sums}
    Let $(X_n)_{n\in F}$ and $(Y_n)_{n\in F}$ be sequences of Banach spaces (lattices), and assume that $T_n\colon X_n\to W_n\subseteq Y_n$ is an isomorphism (Banach lattice isomorphism) for each $n\in F$, where each $W_n$ is a closed subspace of $Y_n$. If $\sup_{n\in F}\|T_n\|<\infty$ and $\sup_{n\in F}\|T_n^{-1}\|<\infty$, 
    then $\ell_\infty((X_n)_{n\in F})$ is isomorphic (Banach lattice isomorphic) to $\ell_\infty((W_n)_{n\in F})\subseteq\ell_\infty((Y_n)_{n\in F})$
    via the map $T\colon(x_n)\mapsto(T_n(x_n))$. In particular, if $\{K_n\,\colon\,n\in\mathbb N\}$ is a partition into infinite subsets of $\mathbb N$, then $\ell_\infty((\ell_\infty(K_m))_{m\in F})$ is isometric to $\ell_\infty$, since $\ell_\infty(K_n)\cong\ell_\infty$ for each $n\in F$ and $\ell_\infty((\ell_\infty)_{n\in F})\cong\ell_\infty$.
\end{remark}

\begin{lemma}
\label{superior limit of sums}
Let $\{K_m\,\colon\,m\in F\}$ be a partition into infinite subsets of $\mathbb N$, where $F\subseteq\mathbb N$, and $\mathcal I_m$ be a proper ideal on $K_m$ for each $m\in F$.
 If ${\bf x}=(x_n)\in\ell_\infty$, then 
    \begin{gather*}
       \left(\bigoplus_{m\in F}\mathcal I_m\right)-\limsup{\bf x}=\sup_{m\in F}\bigg(\mathcal I_m-\limsup {\bf x}\restriction_{K_m}\bigg).
    \end{gather*}
\end{lemma}

\begin{proof}
Let $\mathcal I\coloneqq\bigoplus_{m\in F}\mathcal I_m$ and $A\in\mathbb R$ be such that $\displaystyle\sup_{m\in F}\bigg(\mathcal I_m-\limsup {\bf x}\restriction_{K_m}\bigg)> A$. We have
  \begin{align*}
      \sup_{m\in F}\,\bigg(\mathcal I_m-\limsup {\bf x}\restriction_{K_m}\bigg)> A&\Longleftrightarrow(\exists\,m\in F)(\mathcal I_m-\limsup {\bf x}\restriction_{K_m}> A)\\
&\Longleftrightarrow(\exists\,m\in F)(\{n\in K_m\,\colon\,x_n>A\}\not\in\mathcal I_m)\\
&\Longleftrightarrow\{n\in\mathbb N\,\colon\,x_n>A\}\not\in\mathcal I.
  \end{align*}
  Thus, $\mathcal I-\limsup{\bf x}\geq A$. Since $A$ was arbitrary, we conclude that 
  \begin{gather}\label{first}
      \mathcal I-\limsup{\bf x}\geq\sup_{m\in F}\,\bigg(\mathcal I_m-\limsup x\restriction_{K_m}\bigg).
  \end{gather}
Now fix $\varepsilon>0$. Then, $\big(\mathcal I-\limsup{\bf x}\big)-\varepsilon<b$, where $b\in\mathbb R$ satisfies  $\{n\in\mathbb N\,\colon\,x_n>b\}\not\in\mathcal I$. So, for some $m\in F$ we have that $\{n\in K_m\,\colon\,x_n>b\}\not\in\mathcal I_m$. Then, $\mathcal I_m-\limsup{\bf x}\restriction_{K_m}\geq b$.  It follows that 
   \begin{gather}\label{second}
      \mathcal I-\limsup{\bf x}\leq\sup_{m\in F}\,\bigg(\mathcal I_m-\limsup{\bf x}\restriction_{K_m}\bigg).
  \end{gather}
From \eqref{first} and \eqref{second} we obtain the result.
\end{proof}

\begin{theorem}
\label{quotient of a sum}
Let $\{K_m\,\colon\,m\in F\}$ be a partition into infinite subsets of $\mathbb N$, where $F\subseteq\mathbb N$, $\mathcal I_m$ be a proper ideal on $K_m$ for each $m\in F$ and $\mathcal I=\bigoplus_{m\in F}\mathcal I_m$.  The map 
    \begin{equation*}
          \tilde \xi\colon{\bf x}+c_{0,\mathcal I}\in\ell_\infty/c_{0,\mathcal I}\mapsto ({\bf x}\restriction_{K_m}+c_{0,\mathcal I_m})\in\ell_\infty((\ell_\infty(K_n)/c_{0,\mathcal I_n})_{n\in F})
    \end{equation*}
     is a Banach lattice isometry. 
\end{theorem}
     
\begin{proof}
Let 
\begin{equation*}
    \xi\colon{\bf x}=(x_n)\in\ell_\infty\mapsto({\bf x}\restriction_{K_m}+c_{0,\mathcal I_m})\in\ell_\infty((\ell_\infty(K_m)/c_{0,\mathcal I_m})_{m\in F}).
\end{equation*}
Notice that $\xi$ is linear, continuous and preserves the lattice operations. Since $\xi$  is a composition of onto linear operators, we conclude that $\xi$ is onto. Now, 
\begin{align*}
\xi({\bf x})={\bf0}&\Longleftrightarrow(\forall\,m\in F)({\bf x}\restriction_{K_m}\in c_{0,\mathcal I_m})\\
    &\Longleftrightarrow(\forall\,m\in F)(\forall\,\varepsilon>0)(A(\varepsilon,{\bf x}\restriction_{K_m})\in \mathcal I_m)\\
    &\Longleftrightarrow(\forall\,\varepsilon>0)(\forall\,m\in F)(A(\varepsilon,{\bf x})\cap K_m\in \mathcal I_m)\\
    &\Longleftrightarrow(\forall\,\varepsilon>0)(A(\varepsilon,{\bf x})\in \mathcal I).
\end{align*}
Therefore, $\ker\xi=c_{0,\mathcal I}$. Finally, by Lemma \ref{superior limit of sums} and Remark \ref{norm in the quotient}  we have
\begin{align*}
    \|{\bf x}+\ker\xi\|&=\mathcal I-\limsup|{\bf x}|\\
    &=\sup_{m\in F}\bigg(\mathcal I_m-\limsup|{\bf x}\restriction_{K_m}|\bigg)\\
    &=\sup_{m\in F}\|{\bf x}\restriction_{K_m}+c_{0,\mathcal I_m}\|\\
    &=\|\xi({\bf x})\|.
\end{align*}
The above equation says that the quotient map induced by $\xi$, that is $\tilde\xi$, is an isometry from $\ell_\infty/c_{0,\mathcal I}$ onto $\ell_\infty((\ell_\infty(K_m)/c_{0,\mathcal I_m})_{m\in F}).$
\end{proof}

Now, based on Theorem \ref{teorema michael-saens} we prove that the sum of complemented ideals is complemented. 

\begin{theorem}
\label{comple-discrete}
Let $\{K_m\,\colon\,m\in F\}$ be a partition into infinite subsets of $\mathbb N$ and let $\mathcal I_m$ be an ideal on $K_m$ for each $m\in F$, where $|F|\leq\omega$. Suppose that $c_{0,\mathcal I_m}$ is complemented in $\ell_\infty(K_m)$ for each $m\in F$. Then, the ideal $\mathcal J=\bigoplus_{n\in F}\mathcal I_n$ is complemented. 
\end{theorem}

\begin{proof}
    By Theorem \ref{complementation of ideals in countable sets}, for each $m\in F$, $\ell_\infty(K_m)/c_{0,\mathcal I_m}$ is isometric to a closed subspace of $\ell_\infty$. On the other hand, $\ell_\infty/c_{0,\mathcal I}\cong\ell_\infty((\ell_\infty(K_m)/c_{0,\mathcal I_m})_{m\in F})$ by Theorem \ref{quotient of a sum}. From Remark \ref{l_infty-sums}, we conclude that $\ell_\infty/c_{0,\mathcal I}$ is isometric to a closed subspace of $\ell_\infty((\ell_\infty)_{n\in F})\cong\ell_\infty$. It follows that $c_{0,\mathcal I}$ is complemented in $\ell_\infty$ by Theorem \ref{teorema michael-saens}.
\end{proof}

\begin{corollary}
    Let $\mathcal I$ be an ideal on $\mathbb N$. Then $\mathcal I$ is complemented if and only if $\mathcal I^\omega$ is complemented.
\end{corollary}

\begin{proof}
 If $\mathcal I$ is complemented, then $\mathcal I^\omega$ is complemented by Theorem \ref{comple-discrete}. Conversely, suppose that
 $\mathcal I^\omega$ is complemented. By \cite[Theorem 5.2]{rincon-uzcategui} we have $c_{0,\mathcal I^\omega}\cong\ell_\infty(c_{0,\mathcal I})$. Consequently,
 $c_{0,\mathcal I}$ is isomorphic to a complemented subspace of $c_{0,\mathcal I^\omega}$. On the other hand, by Lindenstrauss's theorem,  $c_{0,\mathcal I^\omega}\sim\ell_\infty$.  Hence,  $c_{0,\mathcal I}$ is isomorphic to a complemented subspace of $\ell_\infty$. Therefore, by Lindenstrauss's theorem, $c_{0,\mathcal I}$ is isomorphic to $\ell_\infty$, and thus $\mathcal I$ is complemented.
\end{proof}

Now we will prove that the intersection of complemented ideals is complemented. 

 \begin{lemma}\label{isometry intersection}
     Let $\{\mathcal I_n\,\colon\,n\in F\}$ be a collection of proper ideals with $|F|\leq\omega$ and $\mathcal I=\bigcap_{n\in F}\mathcal I_n$. 
     The map 
    \begin{equation*}
          \tilde A\colon{\bf x}+c_{0,\mathcal I}\in\ell_\infty/c_{0,\mathcal I}\mapsto ({\bf x}+c_{0,\mathcal I_n})\in\ell_\infty((\ell_\infty/c_{0,\mathcal I_n})_{n\in F})
    \end{equation*}
     is a Banach lattice isometry. 
 \end{lemma}

 \begin{proof}
We start by proving:

\begin{claim}\label{limsup of an intersection}
    If ${\bf x}=(x_n)$ is a bounded sequence in $\mathbb R$, then
     \begin{equation*}
         \mathcal I-\limsup{\bf x}=\sup_{n\in F}\bigg(\mathcal I_n-\limsup{\bf x}\bigg).
     \end{equation*}
\end{claim}

\proof[Proof of the claim] Let $A\in\mathbb R$ be such that $\displaystyle\sup_{n\in F}\bigg(\mathcal I_n-\limsup{\bf x}\bigg)>A$. Hence, $\displaystyle\mathcal I_N-\limsup{\bf x}>A$ for some $N\in F$. The definition of $\mathcal I_N-\limsup$ states that there is $b\in\mathbb R$ such that $b>A$ and $\{k\in\mathbb N\,\colon\,x_k>b\}\not\in\mathcal I_N$. Since
     $\{k\in\mathbb N\,\colon\,x_k>b\}\not\in\mathcal I$, we conclude that $\displaystyle\mathcal I-\limsup{\bf x}\geq b>A$. The arbitrariness of $A$ implies that
     \begin{equation*}
         \displaystyle\mathcal I-\limsup{\bf x}\geq\sup_{n\in F}\bigg(\mathcal I_n-\limsup{\bf x}\bigg).
     \end{equation*}
     To prove the reverse inequality, let $\varepsilon>0$ be given. Then, there exists $b\in\mathbb R$ such that $\displaystyle\mathcal I-\limsup{\bf x}-\varepsilon<b$ and $\{k\in\mathbb N\,\colon\,x_k>b\}\not\in\mathcal I$. So, $\{k\in\mathbb N\,\colon\,x_k>b\}\not\in\mathcal I_n$ for some $n\in F$. Therefore, 
\begin{equation*}
    \displaystyle\mathcal I-\limsup{\bf x}-\varepsilon<b\leq\mathcal I_n-\limsup{\bf x}\leq\sup_{n\in F}\bigg(\mathcal I_n-\limsup{\bf x}\bigg).
\end{equation*}
Since $\varepsilon>0$ was arbitrary, 
\begin{equation*}
         \displaystyle\mathcal I-\limsup{\bf x}\leq\sup_{n\in F}\bigg(\mathcal I_n-\limsup{\bf x}\bigg).\qedhere
     \end{equation*}
\endproof
 
  Now, let $A\colon\ell_\infty\to\ell_\infty((\ell_\infty/c_{0,\mathcal I_n})_{n\in F})$ be defined by $A({\bf x})=({\bf x}+c_{0,\mathcal I_n})$, ${\bf x}\in\ell_\infty$. Then, $A$ is linear, continuous and preserves the lattice operations, and $\ker A=c_{0,\mathcal I}$. Notice that if $\pi\colon\ell_\infty\to\ell_\infty/c_{0,\mathcal I}$ is the quotient map, then
  $\tilde A\circ\pi=A$. Thus, $\tilde A$ is linear, continuous and preserves the lattice operations.
  
  \medskip
  
  If ${\bf x}=(x_n)\in\ell_\infty$ and $|{\bf x}|=(|x_n|)$, then from \cite[Lemma 5.7]{rincon-uzcategui} we have that
  $\|{\bf x}+c_{0,\mathcal I}\|=\mathcal I-\limsup|{\bf x}|$ and $\|{\bf x}+c_{0,\mathcal I_n}\|=\mathcal I_n-\limsup|{\bf x}|$ for each $n\in F$. It follows from Claim \ref{limsup of an intersection} that
 \begin{align*}
        \|\tilde A({\bf x}+c_{0,\mathcal I})\|=\|A({\bf x})\|&=\sup_{n\in F}\|{\bf x}+c_{0,\mathcal I_n}\|\\
    & =\sup_{n\in F}\bigg(\mathcal I_n-\limsup|{\bf x}|\bigg)
     =\mathcal I-\limsup|{\bf x}|
     =\|{\bf x}+c_{0,\mathcal I}\|.\qedhere
 \end{align*} 
 \end{proof}

\begin{theorem}
\label{intersection of complemented ideals is complemented}
Let $|F|\leq\omega$ and $\{\mathcal I_n\,\colon\,n\in F\}$ be a collection of complemented ideals. Then, $\mathcal I=\bigcap_{n\in F}\mathcal I_n$ is complemented. In particular, every $\omega$-maximal ideal is complemented.
\end{theorem}

\begin{proof}
From Lemma \ref{isometry intersection} we obtain that $\ell_\infty/c_{0,\mathcal I}$ is isomorphic to a closed subspace of $\ell_\infty((\ell_\infty/c_{0,\mathcal I_n})_{n\in F})$. On the other hand, since $\mathcal I_n$ is complemented for each $n\in F$, it follows from Theorem \ref{teorema michael-saens}, that $\ell_\infty/c_{0,\mathcal I_n}$ is isometric to a closed subspace of $\ell_\infty$ for every $n\in F$. By Remark \ref{l_infty-sums}, there exists an isometry from $\ell_\infty((\ell_\infty/c_{0,\mathcal I_n})_{n\in F})$ onto a closed subspace of $\ell_\infty((\ell_\infty)_{n\in F})\cong\ell_\infty$. Consequently, $\ell_\infty/c_{0,\mathcal I}$ is isometric to a closed subspace of $\ell_\infty$. By Theorem \ref{teorema michael-saens}, $\mathcal I$ is complemented.
\end{proof}

\begin{remark}
   Example 2 in \cite{hrusak-saenz2025} shows that  the second part of Theorem~\ref{intersection of complemented ideals is complemented} is not an equivalence.
\end{remark}

 The next result shows that the complementation of ideals is preserved by taking restrictions.

\begin{proposition}\label{restricciones complementados}
Let $\mathcal I$ be an ideal on $\mathbb N$, and let $A$ and $B$ be infinite subsets of $\mathbb N$ with $A\subseteq B$ and $A\in\mathcal I^+$. If $\mathcal I\restriction B$ is complemented in $B$, then $\mathcal I\restriction A$ is complemented in $A$. In particular, if $\mathcal I$ is complemented, then $\mathcal I\restriction A$ is complemented in $A$ for every $A\in\mathcal I^+$.
\end{proposition}

\begin{proof}
   Since $c_{0,\mathcal I\restriction B}$ is complemented in $\ell_\infty(B)$, it suffices to prove that $c_{0,\mathcal I\restriction A}$ is isometric to a complemented subspace of $c_{0,\mathcal I\restriction B}$.
  To this end, consider the map $ \varphi\colon c_{0,\mathcal I\restriction A}\to c_{0,\mathcal I\restriction B}$ defined by
  \begin{gather*}
      (x_n)\mapsto(\tilde{x}_n),\quad\mbox{where}\quad \tilde{x}_n=\begin{cases}
           x_n, & n\in A;\\
           0, & n\in B\setminus A.
       \end{cases}
  \end{gather*}
This mapping is an isometry. Furthermore, the map $ P\colon c_{0,\mathcal I\restriction B}\to\varphi(c_{0,\mathcal I\restriction A})$ given by
  \begin{gather*}
          (x_n)\mapsto(y_n),\quad\mbox{where}\quad y_n=\begin{cases}
           x_n, & n\in A;\\
           0, & n\in B\setminus A,
       \end{cases}
  \end{gather*}
defines a projection, completing the proof.
\end{proof}

\section{At most $\omega$-maximal ideals and strongly $\omega$-maximal ideals}\label{at most maximal}

Recall that an ideal $\ideal$ is called at most $\omega$-maximal if it is $k$-maximal for some positive integer $k$ or it is $\omega$-maximal. Also, we called it strongly $\omega$-maximal when it is an intersection of countably infinite collection $\{\ideal_m: m\in\mathbb 
N\}$ of maximal ideals such that $\{\ideal_m^*:m\in \N\}$ is a discrete subset of $\beta\N$. By the results established in Section \ref{complementacion}, we know that these type of ideals are complemented. However, it remains unclear when the complementation determines the structure of the ideal. In this section we study the interplay between the complementation of $c_{0,\mathcal I}$ and the structure of the ideal $\mathcal I$. 

We will show that, under an additional condition, the complementation of $c_{0,\mathcal I}$ is equivalent to $\ideal$ being at most $\omega$-maximal. We further isolate a new  subclass of ideals which we call \textit{discretely represented}  which are complemented and the witnessing operator can be taken to be an explicit positive projection respecting the Boolean algebra of $2^\mathbb N$. Finally, we characterize strongly $\omega$-maximal ideals in terms of the quotient $\ell_\infty/c_{0,\mathcal I}$.

\medskip

We start by stating a key lemma whose argument is motivated by the proof of  the non-complementation of $c_0$ in $\ell_\infty$ due to Whitley (see \cite[Theorem 5.6]{fabian et al} or \cite[Theorem 2.5.4]{AK}). 

\begin{lemma}
\label{esquema2}
Let $\mathcal I\subsetneq \idealj$ be  ideals on $\mathbb N$. Suppose that there is a linear bounded operator $Q\colon c_{0,\mathcal J}\to c_{0,\mathcal J}$ such that $\ker Q=c_{0,\mathcal I}$. Then, there exists a collection $\{\mathcal A_{n,k}\colon\,n,k\in\mathbb N\}$ of subsets of $2^\mathbb N$ satisfying:
\begin{enumerate}
\item $\mathcal A_{n,k}\subseteq \idealj\cap  \mathcal I^+$ for each $n,k\in\mathbb N$. 

\item $\mathcal A_{n,k}\subseteq\mathcal A_{n,k+1}$ for all $n,k\in\mathbb N$.

\item For each $n\in\mathbb N$, let $\mathcal{A}_n=\bigcup_{k\in\mathbb N}\mathcal A_{n,k}$. Then $\idealj \cap \mathcal I^+=\bigcup_{n\in\mathbb N}\mathcal A_{n}$.

\item  If $A\in\mathcal A_{n,k}$ and $A=B\cup C$ with $B\cap C\not\in\mathcal{A}_n$, then either $B\in\mathcal A_{n,2k}$ or $C\in\mathcal A_{n,2k}$. In particular, if $A\in\mathcal A_{n}$ and $A=B\cup C$ with $B\cap C\not\in\mathcal{A}_n$, then either $B\in\mathcal A_{n}$ or $C\in\mathcal A_{n}$.
\end{enumerate}
\end{lemma}

\begin{proof}
 For each $n,k\in\mathbb N$, define 
\begin{gather*}
    \mathcal A_{n,k}=\{A\in \idealj\,\colon\,|Q(\chi_A)_n|\geq 1/k\}.
\end{gather*}
Since $\mathcal I=\{A\in \idealj \,\colon\,Q(\chi_A)={\bf0}\}$, we have
    \begin{align*}
         \idealj\cap \mathcal I^+&=\{A\in  \idealj\,\colon\,Q(\chi_A)\neq{\bf0}\}\\
         &=\bigcup_{k,n\in\mathbb N}\{A\in \idealj\,\colon\,|Q(\chi_A)_n|\geq1/k\}\\
         &=\bigcup_{n\in\mathbb N}\mathcal A_{n}.
    \end{align*}
Then (1) and (3) hold. (2) is obvious. 

Finally, we show (4). Let $A\in\mathcal A_{n,k}$ and $A=B\cup C$ with $B\cap C\not\in\mathcal{A}_n$. Then $|Q(\chi_{B\cap C})_n|=0$. Assume that $|Q(\chi_B)_n|<1/2k$ and $|Q(\chi_C)_n|<1/2k$. Note that $Q(\chi_B)_n= Q(\chi_{B\setminus C})_n$ and $Q(\chi_C)_n=Q(\chi_{C\setminus B})_n$. Thus
\begin{gather*}
    |Q(\chi_A)_n|=|Q(\chi_{A\setminus B\cap C})_n|=|Q(\chi_B)_n+Q(\chi_C)_n|<1/k,
\end{gather*}
which is absurd.
\end{proof}

Now we prove, with an extra assumption, that the complementation of $c_{0,\mathcal I}$ is equivalent to $\ideal$ being at most $\omega$-maximal. 

\begin{theorem}
\label{implicaciones}
Let $\ideal$ be a proper ideal on $\N$. Then, the following statements are equivalent:
\begin{enumerate}
\item $\mathcal I$ is  at most $\omega$-maximal.

\item There is a Banach lattice isomorphism from $\ell_\infty/c_{0,\mathcal I}$ onto a closed sublattice of $\ell_\infty$.

 \item There exists a bounded linear operator  $Q\colon\ell_\infty\to\ell_\infty$ with $\ker Q=c_{0,\mathcal I}$ that verifies the next property: for every pair of subsets $A,B\subseteq\N$ and every $n\in \N$,
    \begin{equation}\label{Q-mult}
Q(\chi_A)_n\neq0\quad\mbox{and}\quad Q(\chi_B)_n\neq0\;\Rightarrow \; Q(\chi_{A\cap B})_n\neq0. 
\end{equation}
\end{enumerate}
\end{theorem}

\proof 
$(1)\Rightarrow(2)$ Let $\{\mathcal I_n\,\colon\,n\in F\}$ be a collection of maximal ideals where  $|F|\leq\omega$ and $\mathcal{I}=\bigcap_{n\in F}\mathcal{I}_n$. Notice that $\ell_\infty/c_{0,\mathcal I_n}$ is Banach lattice isometric to $\mathbb R$ for each $n\in F$ via the map ${\bf x}+c_{0,\mathcal I_n}\mapsto\mathcal I_n-\lim{\bf x}$. Now, if $|F|=k$, the sum $\ell_\infty((\ell_\infty/c_{0,\mathcal I_n})_{n\in F})$ is Banach lattice isometric to $(\mathbb R^k,\|\cdot\|_\infty)$, which is Banach lattice isometric to a closed sublattice of $\ell_\infty$. If $|F|=\omega$, then $\ell_\infty((\ell_\infty/c_{0,\mathcal I_n})_{n\in F})$ is Banach lattice isometric to $\ell_\infty$. On the other hand, Lemma \ref{isometry intersection} says that the map ${\bf x}+c_{0,\mathcal I}\in\ell_\infty/c_{0,\mathcal I}\mapsto({\bf x}+c_{0,\mathcal I_n})\in\ell_\infty((\ell_\infty/c_{0,\mathcal I_n})_{n\in F})$ is a Banach lattice isometry. Thus, there is a Banach lattice isomorphism from $\ell_\infty/c_{0,\mathcal I}$ onto a closed sublattice of $\ell_\infty$.

\medskip

$(2)\Rightarrow(3)$  Let $T$ be a Banach lattice isomorphism from $\ell_\infty/c_{0,\mathcal I}$ onto a closed sublattice of $\ell_\infty$. Define $Q\colon\ell_\infty\to\ell_\infty$ by 
$Q({\bf x})=T({\bf x}+c_{0,\mathcal I})$ for each ${\bf x}\in\ell_\infty$. It is clear that $Q$ is linear and bounded. Since $T$ is injective, we have:
\begin{equation*}
    Q({\bf x})={\bf 0}\Longleftrightarrow T({\bf x}+c_{0,\mathcal I})={\bf 0}\Longleftrightarrow {\bf x}\in c_{0,\mathcal I}.
\end{equation*}
Thus, $\ker Q=c_{0,\mathcal I}$.

We will verify \eqref{Q-mult}. 
Let $n\in\mathbb N$ and $A,B\subseteq\mathbb N$ be such that $Q(\chi_A)_n\neq0$ and $Q(\chi_B)_n\neq0$. Observe that $Q(\chi_A)=T(\chi_A+c_{0,\mathcal I})\geq{\bf0}$ and
$Q(\chi_B)=T(\chi_B+c_{0,\mathcal I})\geq{\bf0}$. Hence, $Q(\chi_A)_n=T(\chi_A+c_{0,\mathcal I})_n>0$ and $Q(\chi_B)_n=T(\chi_B+c_{0,\mathcal I})_n>0$. From the fact $(\chi_A+c_{0,\mathcal I})\wedge(\chi_B+c_{0,\mathcal I})=\chi_{A\cap B}+c_{0,\mathcal I}$, it follows that
$T(\chi_{A\cap B}+c_{0,\mathcal I})=T(\chi_A+c_{0,\mathcal I})\wedge T(\chi_B+c_{0,\mathcal I})$ and 
\begin{equation*}
    Q(\chi_{A\cap B})_n=T(\chi_{A\cap B}+c_{0,\mathcal I})_n=T(\chi_A+c_{0,\mathcal I})_n\wedge T(\chi_B+c_{0,\mathcal I})_n>0.
\end{equation*}
Thus, $Q(\chi_{A\cap B})_n\neq0$.

\medskip

$(3)\Rightarrow(1)$ Let $Q\colon\ell_\infty\to \ell_\infty$ be a bounded linear operator with $\ker Q=c_{0,\mathcal I}$ and satisfying \eqref{Q-mult}. For each $n\in\mathbb N$, consider the family 
$$
\mathcal{A}_n=\{A\subseteq \N: \; Q(\chi_A)_n\neq 0\}
$$
as in Lemma \ref{esquema2} with $\idealj=\mathcal P (\N)$. Observe that $\ideal^+=\bigcup_{n\in\mathbb N} \mathcal{A}_n$ by Lemma \ref{esquema2}. For each $n\in\mathbb N$, let  
\begin{gather*}
    \mathcal{H}_n=\{B\subseteq \N:(\exists\, A\in \mathcal{A}_n)(A\subseteq B)\}.
\end{gather*}
Notice that $\mathcal{H}_n\subseteq \ideal^+$ for all $n$. Indeed, suppose not and let $B\in\mathcal H_n$ with $B\not\in\mathcal I^+$.  Let $A\in\mathcal A_n$ be such that $A\subseteq B$. Since $B\in\mathcal I$, $A\in\mathcal I$ and thus $Q(\chi_A)={\bf0}$, which is absurd. 

We claim that $\mathcal I_n\coloneqq 2^{\N}\setminus \mathcal{H}_n$ is an ideal. By construction, $\mathcal I_n$ is closed by taking subsets. Now, let $B,C\in\mathcal I_n$ and suppose that $B\cup C\not\in\mathcal I_n$, that is, $B\cup C\in\mathcal H_n$.  Let $A\in\mathcal A_n$ be such that $A\subseteq B\cup C.$ Observe that $A\cap(B\cap C)=(A\cap B)\cap (A\cap C)\not\in \mathcal{A}_n$, since $A\cap(B\cap C)\subseteq B$ and $B\in\mathcal I_n$. By applying (5) of Lemma \ref{esquema2} to the decomposition $A=(A\cap B)\cup(A\cap C)$, we have either $A\cap B\in\mathcal A_n$ or $A\cap C\in\mathcal A_n$, and consequently either $B\in\mathcal H_n$ or $C\in\mathcal H_n$, a contradiction. So, $\mathcal I_n$ is an ideal. 

As $\mathcal I$ is proper, $\mathbb N\not\in\mathcal I$. Thus, 
$S=\{n\in\mathbb N\,\colon\,Q(\chi_{\mathbb N})_n\neq0\}$ is non-empty. Clearly, $\mathbb N\in\mathcal A_n$ if and only if $n\in S$. Whence, $n\in S$ implies that $\mathcal I_n\neq2^\mathbb N$. To prove the converse, if $\mathcal I_n\neq2^\mathbb N$, we have $\mathcal A_n\neq\emptyset$. Let $A\in\mathcal A_n$. If $Q(\chi_{\mathbb N})_n=0$, then $Q(\chi_{\mathbb N\setminus A})_n=-Q(\chi_A)_n\neq0$. From \eqref{Q-mult}, we obtain $0=Q({\bf0})_n=Q(\chi_{A\cap\mathbb N\setminus A})_n\neq0$ which is impossible. Therefore, $\mathcal I_n\neq2^\mathbb N$ if and only if $n\in S$.

By (4) of Lemma \ref{esquema2}, $\ideal=\bigcap_n \ideal_n=\bigcap_{n\in S}\mathcal I_n$. It remains to show that each $\ideal_n$ is maximal for each $n\in S$. If not, let $B\subseteq \N$ be such that  $B\not\in\ideal_n$ and $\N\setminus B\not\in \ideal_n$. Then, there are $C_0,C_1\in \mathcal{A}_n$ such that $C_0\subseteq B$ and $C_1\subseteq \N\setminus B$. By \eqref{Q-mult} we have $Q(\chi_{C_0\cap C_1})_n\neq 0$. On the other hand, $Q(\chi_{C_0\cap C_1})={\bf0}$, an absurd.

As $|S|\leq \omega$, $\ideal$ is at most $\omega$-maximal. This completes the proof. 
\endproof

\begin{remark}
The Property \eqref{Q-mult} cannot be omitted. Indeed, by Corollary \ref{complementacion es equivalente a operador con kernel igual al ideal}, the existence of a bounded linear operator $Q$ from $\ell_\infty$ onto a closed subspace of $\ell_\infty$ with $\ker Q=c_{0,\mathcal I}$ implies that $\mathcal I$ is complemented. However, \cite[Example 2]{hrusak-saenz2025} shows  a complemented ideal which is not at most $\omega$-maximal. 
\end{remark}

\medskip

Observe that if $\mathcal I$ is a strongly $\omega$-maximal ideal, then  $\mathcal I$ is complemented by Theorem \ref{intersection of complemented ideals is complemented}. Our next result provides an alternative proof by constructing an explicit projection. To this end, we first establish an auxiliary result. 

Recall that if ${\bf x}=(x_n)\in\ell_\infty$, its support is defined as
$\mathrm{supp}({\bf x})\coloneqq\{n\in\mathbb N\,\colon\,x_n\neq0\}$. Let $F\subseteq\mathbb N$, and suppose that $({\bf x}_n)_{n\in F}$ is a sequence in $\ell_\infty$ such that $\mathrm{supp}({\bf x}_m)\cap\mathrm{supp}({\bf x}_n)=\emptyset$ for all $m\neq n$. Then, the sum $\sum_{n\in F}{\bf x}_n$ is the sequence in $\ell_\infty$ whose $m$-th coordinate is ${\bf x}_n(m)$ whenever $m\in\mathrm{supp}({\bf x}_n)$ for some $n\in F$, and 0 otherwise. Observe that the previous sum is well defined since the supports are pairwise disjoint.

\begin{lemma}
\label{construccion de la proyeccion}
Let $|F|\leq\omega$ and $\{\mathcal I_n\,\colon\,n\in F\}$ be a collection of maximal ideals, and let $\{A_n\,\colon\,n\in F\}$ be a collection of pairwise disjoint subsets of $\mathbb N$ such that $A_n\in\mathcal I_n^*$ for each $n\in F$. Let $\mathcal I=\bigcap_{n\in F}\mathcal I_n$. Then 
    \begin{align*}
        P\colon \ell_\infty&\to c_{0,\mathcal I}\\
        {\bf x}&\mapsto{\bf x}-\sum_{m\in F}(\mathcal I_m^*-\lim{\bf x})\chi_{A_m}
    \end{align*}
    is a continuous projection from $\ell_\infty$ onto $c_{0,\mathcal I}$.
\end{lemma}

\begin{proof}
 Firstly, we prove that $P$ is well defined. Let ${\bf x}=(x_n)\in\ell_\infty$ and ${\bf y}=P({\bf x})=(y_n)$. Let  $\varepsilon>0$ and $m\in F$. We have
\begin{gather*}
     A(\varepsilon,{\bf y})\cap A_m=\{n\in A_m\,\colon\,|y_n|\geq\varepsilon\}
    =\{n\in A_m\,\colon\,|x_n-(\mathcal I_m^*-\lim{\bf x})|\geq\varepsilon\}.
\end{gather*}
From the definition of $\mathcal I_m^*-\lim{\bf x}$ it follows that $\{n\in\mathbb N\,\colon\,|x_n-(\mathcal I_m^*-\lim{\bf x})|\geq\varepsilon\}\in\mathcal I_m$. 
Thus, $A(\varepsilon,{\bf y})\cap A_m\in\mathcal I_m$ for each $\varepsilon>0$ and $m\in F$. By Lemma \ref{suma-directa}, $A(\varepsilon,{\bf y})\in\mathcal I$ for every $\varepsilon>0$. We conclude that ${\bf y}=P({\bf x})\in c_{0,\ideal}$. Therefore, $P$ is well defined. 

Clearly, $P$ is linear and $\|P\|\leq2$. To check that $P$ is a projection, let ${\bf x}\in\ell_\infty$. We have
\begin{gather*}
    P(P({\bf x}))=P({\bf x})-\sum_{m\in F}(\mathcal I_m^*-\lim P({\bf x}))\chi_{A_m}=P({\bf x}),
\end{gather*}
because of $P({\bf x})\in c_{0,\mathcal I_m}$, that is, $\mathcal I_m^*-\lim P({\bf x})=0$ for each $m\in F$. The previous argument also shows that $P(\ell_\infty)=c_{0,\mathcal I}$. Hence, $P$ is a projection from $\ell_\infty$ onto $c_{0,\mathcal I}$.
\end{proof}

We say that an ideal $\ideal$ is {\em discretely represented} if it is either $k$-maximal for some positive integer $k$ or it is strongly $\omega$-maximal. Next result presents an explicit projection  for this type of ideals.

\begin{theorem}
\label{proyecciones1}
Let $\ideal$ be a  discretely represented ideal. Then, there is a positive projection  $Q\colon\ell_\infty\to\ell_\infty$ such that $\ker Q=c_{0,\mathcal I}$. Moreover, if for each $B\subseteq\mathbb N$ we let $T(B)=\{n\in\mathbb N\,\colon\,Q(\chi_B)_n=1\}$, then the following properties hold:

\begin{enumerate}
\item $Q(\chi_A)=\chi_{T(A)}$ for all $A\subseteq\mathbb N$ and $Q(\chi_{\mathbb N})=\chi_{\mathbb N}$;

\item If $A\subseteq B$, then $T(A)\subseteq T(B)$;
    
\item $T(T(A))=T(A)$ for all $A\subseteq\mathbb N$;
    
\item $T(A\cap B)=T(A)\cap T(B)$ for all $A,B\subseteq\mathbb N$;
    
\item The family $\mathcal B=\{B\subseteq\mathbb N\,\colon\,T(B)=B\}$ is closed under arbitrary intersections.

 \item If $\mathcal I$ is strongly $\omega$-maximal, $Q$ has infinite rank.
    
\end{enumerate}
\end{theorem}

\begin{proof} Let $\{\mathcal I_j\,\colon\,j\in F\}$, with $|F|\leq \omega$ be a discrete collection of pairwise distinct maximal ideals on $\mathbb N$ such that  $\mathcal I=\bigcap_{j\in F}\mathcal I_j$. By Proposition \ref{particiones2}, there is a partition  $\{A_m\,\colon\,m\in F\}$ of $\mathbb N$ with $A_m\in\mathcal I_m^*$ for each $m\in F$. Consider
the projection defined in Lemma \ref{construccion de la proyeccion}, that is, 
\begin{align*}
        P\colon \ell_\infty&\to c_{0,\mathcal I}\\
        {\bf x}&\mapsto{\bf x}-\sum_{m\in F}(\mathcal I_m^*-\lim{\bf x})\chi_{A_m}.
    \end{align*}
 Now define $Q\colon\ell_\infty\to\ell_\infty$ by $Q=Id-P$, that is, 
\begin{equation}\label{form of the projection}
Q({\bf x})=\sum_{m\in F}(\mathcal I_m^*-\lim{\bf x})\chi_{A_m},\quad {\bf x}\in\ell_\infty.
\end{equation}
 Notice that $Q$ is a positive projection and $\ker Q= c_{0,\mathcal I}$. Now we will check properties (1)-(5). 
 
 Observe that (1) follows from the definition of $Q$. Properties (2) and (3) follow from the positiveness and idempotence of $Q$. For (4), notice that if $n\in\N$, then $Q(\chi_A)_n=1$ if and only if there is $m\in F$ such that
 $n\in A_m$ and $A\in\mathcal I_m^*$. Let $A,B\subseteq\mathbb N$ be given. By (2) it suffices to show that $T(A)\cap T(B)\subseteq T(A\cap B)$. Let $n\in T(A)\cap T(B)$, that is, $Q(\chi_A)_n=1$ and $Q(\chi_B)_n=1$. Thus, there is $m\in F$
 such that $n\in A_m$, $A\in\mathcal I_m^*$ and $B\in\mathcal I_m^*$. Therefore, $A\cap B\in\mathcal I_m^*$. So, $n\in T(A\cap B)$. 
 
 Now we will prove (5).
\begin{claim}
\label{A_m se mete en B si lo toca}
Let $B\subseteq\mathbb N$ be such that $T(B)=B$. If $A_m\cap B\neq\emptyset$, then $A_m\subseteq B$.
\end{claim}

Let $k\in A_m$. If $n\in A_m\cap B$, then $n\in T(B)$. Thus, there is $j\in F$ such that $n\in A_j$ and $B\in\mathcal I_j^*$. As $n\in A_m$, $j=m$. 
Hence, $k\in A_m$ and $B\in\mathcal I_m^*$, that is, $k\in T(B)=B$.

\medskip

Let $\{C_i\}_{i\in I}$ be a collection in $\mathcal B$. By (2) we have $T(\bigcap_{i\in I}C_i)\subseteq\bigcap_{i\in I}C_i$. Let $n\in\bigcap_{i\in I}C_i$.
For each $i\in I$, there is $m_i\in F$ such that $n\in A_{m_i}$ and $C_i\in\mathcal I_{m_i}^*$. Observe that for some $m\in F$ we have $m_i=m$ for each $i\in I$. Since $A_m\cap C_i\neq\emptyset$ and $T(C_i)=C_i$ for each $i\in I$, by Claim \ref{A_m se mete en B si lo toca} we obtain $A_m\subseteq C_i$ for all $i\in I$.
Consequently, $A_m\subseteq\bigcap_{i\in I}C_i$. It follows that $\bigcap_{i\in I}C_i\in\mathcal I_m^*$. So, $n\in T(\bigcap_{i\in I}C_i)$ and we are done. 

Finally, suppose $\ideal$ is strongly $\omega$-maximal, i.e., $F$ is countably infinite. Notice that $Q(\chi_{A_m})= \chi_{A_m}$ for each $m\in F$. Since the $A_m$'s are pairwise disjoint, $\{\chi_{A_m}: m\in F\}$ is linearly independent. Thus, $Q$ has infinite rank.
\end{proof}

Our next result is the  converse of  Theorem \ref{proyecciones1}. 

\begin{theorem}
\label{proyecciones1b}
Let $\mathcal I$ be a proper  ideal on $\mathbb N$. Assume that there is an operator $Q\colon\ell_\infty\to\ell_\infty$ such that 

\begin{enumerate}
\item[(A)] $Q$ is  a positive projection such that $\ker Q=c_{0,\mathcal I}$.
    
\item[(B)] For each $A\subseteq\mathbb N$, there exists $T(A)\subseteq\mathbb N$ such that $Q(\chi_A)=\chi_{T(A)}$.

\item[(C)] If $A,B\subseteq\mathbb N$, then $T(A\cap B)=T(A)\cap T(B)$.
    
\item[(D)] The family $\mathcal B=\{B\subseteq\mathbb N\,\colon\,T(B)=B\}$ is closed under arbitrary intersections.
    
\end{enumerate} 
Then, $\mathcal I$ is  discretely represented. Moreover, there is a partition of $\mathbb N$ such that $Q$ has the form given by \eqref{form of the projection}, and in the case where $Q$ has infinite rank, $\mathcal I$ is strongly $\omega$-maximal.
\end{theorem}

\begin{proof} For each $n\in \mathbb N$, let $V_n=\bigcap\{B\in\mathcal B\,\colon\,n\in B\}$. Since $\mathcal B$ is closed under arbitrary intersections, it follows that $V_n\in \mathcal{B}$, i.e., $T(V_n)=V_n$ for every $n\in\mathbb N$.

\begin{claim}\label{properties of T}
The map $T\colon A\in \mathcal P(\mathbb N)\mapsto T(A)\in\mathcal P(\mathbb N)$ has the following properties:

\begin{enumerate}
 \item\label{idempotencia of T} For each $A\subseteq\mathbb N$, $T(T(A))=T(A)$;
 \item\label{monotonia de T} If $A\subseteq B$, then $T(A)\subseteq T(B)$.
 \item $Q(\chi_{\mathbb N})=\chi_{\mathbb N}$. 
\end{enumerate}
\end{claim} 

\proof[Proof of the Claim \ref{properties of T}] Properties \eqref{idempotencia of T} and \eqref{monotonia de T} follow from idempotence and positiveness of $Q$.  For (3), it suffices to show that $\mathbb N
\subseteq T(\mathbb N)$. Let $n\in\mathbb N$ be given. As $n\in V_n$ and $ V_n=T(V_n)\subseteq T(\mathbb N)$, we have $n\in T(\mathbb N)$.
\endproof

We will follow the notation used in the proof of Theorem \ref{implicaciones}. Recall that for each $n\in\mathbb N$, we set $\mathcal A_n=\{A\subseteq\mathbb N\,\colon Q(\chi_A)_n\neq0\}$. As $Q$ is a positive operator, $\mathcal H_n=\mathcal A_n$ for all $n\in\mathbb N$.
Also notice that if $A,B\subseteq\mathbb N$ satisfy that $Q(\chi_A)_n=\chi_{T(A)}(n)\neq0$ and $Q(\chi_B)_n=\chi_{T(B)}(n)\neq0$, then
$n\in T(A)\cap T(B)=T(A\cap B)$, that is, $Q(\chi_{A\cap B})_n=1$. Thus, $Q$ verifies the property \eqref{Q-mult}. Hence, by the proof of Theorem \ref{implicaciones}, we have  $\mathcal I_n=2^\mathbb N\setminus\mathcal A_n$ is a maximal ideal for each $n\in\{m\in\mathbb N\,\colon\,Q(\chi_{\mathbb N})_m\neq0\}=\N$ (using  part (3) of Claim \ref{properties of T}). Also, we have  that  $\mathcal I=\bigcap_{n\in \N}\mathcal I_n$. 

\begin{claim}
\label{V_n=V_m si y solo si A_m=A_n}
For all $n,m\in \N$, $\mathcal{A}_m\subseteq\mathcal{A}_{n}$ if and only if $V_n\subseteq V_m$. In particular, for each $m,n\in\mathbb N$ it holds $\mathcal{A}_m=\mathcal{A}_{n}$ if and only if $V_n=V_m$.
\end{claim}  

\proof[Proof of the Claim \ref{V_n=V_m si y solo si A_m=A_n}] Suppose $\mathcal{A}_m\subseteq \mathcal{A}_{n}$. Since $V_m\in\mathcal A_m$, we have $V_m\in\mathcal A_n$, that is, $n\in T(V_m)=V_m$. Thus, $V_n\subseteq V_m$.

Conversely, suppose $V_n \subseteq V_m$. Let $B\in \mathcal{A}_m$, that is, $m\in T(B)$. Since $T(T(B))=T(B)$, $T(B)\in \mathcal{B}$. Thus, $V_m\subseteq T(B)$. Hence, $V_n\subseteq T(B)$.
In particular, $n\in T(B)$, that is, $B\in \mathcal{A}_n$. Therefore, $\mathcal{A}_m\subseteq \mathcal{A}_n$. 
\endproof

We will show  that if  $V_n\cap V_m\neq \emptyset$, then $V_n=V_m$, for all $m,n\in\mathbb N$. Indeed, let $k\in V_n\cap V_m$. As $V_n\cap V_m\in \mathcal{B}$,   we have that $V_k\subseteq V_m\cap V_n$. By Claim \ref{V_n=V_m si y solo si A_m=A_n}, $\mathcal{A}_n\subseteq \mathcal{A}_k$. Thus, by the maximality of $\mathcal A_n$, $\mathcal{A}_n=\mathcal{A}_k$. Analogously, $\mathcal{A}_m=\mathcal{A}_k$. Hence, $V_n=V_m$. 

Now consider the following equivalence relation on $\mathbb N$:  $m\sim n$ if and only if $\mathcal A_n=\mathcal A_m$. By Claim \ref{V_n=V_m si y solo si A_m=A_n} we have $n\sim m$ if and only if $V_n=V_m$. Moreover, we have shown that $n\not\sim m$ if and only if $V_n\cap V_m=\emptyset$. Let $F$ be a complete set of representatives and consider the collection $\{V_n\,\colon\,n\in F\}$. Notice that if $n,m\in F$ and $n\neq m$, then $V_n\cap V_m=\emptyset$. Also, $V_n\in\mathcal I_n^*=\mathcal A_n$ for each $n\in F$. Thus, $\{\mathcal I_n\,\colon\,n\in F\}$ is discrete. Moreover, since $\mathcal I=\bigcap_{n\in\mathbb N}\mathcal I_n$  and $F$ is a complete set of representatives, we have
$\mathcal I=\bigcap_{n\in F}\mathcal I_n$. 
Consequently, if $F$ is finite, then $\mathcal I$ is  $k$-maximal for $k=|F|$.  Otherwise, $\mathcal I$ is strongly $\omega$-maximal. Hence, $\mathcal I$ is  discretely represented.

Finally,  we will prove that, regardless of the size of $F$, the following holds:
    \begin{gather*}
        Q({\bf x})=\sum_{m\in F}(\mathcal I_m^*-\lim x)\chi_{V_m},\quad\text{for all ${\bf x}\in\ell_\infty$.}
    \end{gather*}
    For all $A\subseteq\mathbb N$ and $n\in\mathbb N$ it holds that $\mathcal I_n^*-\lim\chi_A=\chi_{T(A)}(n)$. If ${\bf x}=\sum_{i=1}^kc_i\chi_{A_i}$, then $ Q({\bf x})_n=\sum_{i=1}^kc_i\chi_{T(A_i)}(n).$ On the other hand, 
   \begin{align*}
       \sum_{m\in F}(\mathcal I_m^*-\lim{\bf x})\chi_{V_{m}}(n)&=(\mathcal I_k^*-\lim{\bf x})\chi_{V_k}(n)\\
       &=(\mathcal I_n^*-\lim{\bf x})\chi_{V_n}(n)\\
       &=\sum_{i=1}^kc_i\chi_{T(A_i)}(n).
   \end{align*}
   Thus, $ Q({\bf x})=\sum_{m\in F}(\mathcal I_m^*-\lim{\bf x})\chi_{V_m}$ when ${\bf x}=\sum_{i=1}^kc_i\chi_{A_i}$. The general case follows since the linear span of the set
   $\{\chi_A\,\colon\,A\subseteq\mathbb N\}$ is dense in $\ell_\infty$ (see \cite[Proposition 2.5]{rincon-uzcategui}). Therefore, $Q$ has the same form as given in \eqref{form of the projection}.

To conclude, if $F$ is finite, then $Q$ has finite rank since $\{\chi_{V_n}\,\colon\,n\in F\}$ is a linearly independent subset of $\ell_\infty$.
\end{proof}

Theorems \ref{proyecciones1} and \ref{proyecciones1b} motivates the following definition: Let $\mathcal I$ be an ideal on $\mathbb N$. A positive projection $Q\colon\ell_\infty\to\ell_\infty$ is called an \emph{$\mathcal I$-discrete projection} if $\ker Q = c_{0,\mathcal I}$ and $Q$ satisfies the boolean properties (B)-(D) of Theorem \ref{proyecciones1b}.

 Consequently, it follows from Theorems \ref{proyecciones1} and \ref{proyecciones1b} that:

\begin{theorem}\label{I-discrete es igual a I-discrete projection}
    Let $\mathcal I$ be an ideal on $\mathbb N$. Then, $\mathcal I$ is discretely represented if and only if there is an $\mathcal I$-discrete projection $Q\colon\ell_\infty\to\ell_\infty$. Moreover, $Q$ has infinite rank if and only if $\mathcal I$ is strongly $\omega$-maximal.
\end{theorem}

Now, in the same spirit of Theorems \ref{teorema michael-saens} and \ref{implicaciones}, we give a characterization of strongly $\omega$-maximal ideals
in terms of the quotient $\ell_\infty/c_{0,\mathcal I}$. We say $\ideal$ is {\em strongly complemented} if there is an $\mathcal I$-discrete projection $Q\colon\ell_\infty\to\ell_\infty$ having infinite rank.

\medskip

Recall that for an ideal $\mathcal I$, $K_{\mathcal I}=\{p\in\beta\mathbb N\,\colon\,\mathcal I^*\subseteq p\}$.

\begin{theorem}\label{s-omega-max caracterizacion in terms of the cociente}
    Let $\mathcal I$ be an ideal on $\mathbb N$. The following statements are equivalent:
    \begin{enumerate}
        \item $K_\mathcal I$ is homeomorphic to $\beta\mathbb N$;
        \item $\ell_\infty/c_{0,\mathcal I}$ is isometric to $\ell_\infty$;
        \item $\ell_\infty/c_{0,\mathcal I}$ is Banach lattice isomorphic to $\ell_\infty$;
        \item $\mathcal I$ is strongly $\omega$-maximal.
        \item $\mathcal I$ is strongly complemented. 
    \end{enumerate}
\end{theorem}

\begin{proof}
    $(1)\Longleftrightarrow(2)$ It is classical result that $C(\beta\mathbb N)$ is isometric to $\ell_\infty$ and $C(K_\ideal)$ is isometric to $\ell_\infty/c_{0,\mathcal I}$ by Theorem \ref{kania}. Thus,  the claim follows from the classical Banach-Stone theorem \cite[Theorem 4.1.5]{AK}.  
    
    $(1)\Longleftrightarrow(3)$ is shown analogously using  a well known result of Kaplansky \cite{kaplansky}.
    
    $(4)\Longleftrightarrow(5)$ is proved in Theorem \ref{I-discrete es igual a I-discrete projection}.

    $(4)\Rightarrow(3)$. Suppose that $\mathcal I$ is strongly $\omega$-maximal. By Proposition \ref{particiones2} and Lemma \ref{suma-directa}, there is a partition
    $\{A_n\,\colon\,n\in\mathbb N\}$ such that $A_n\in\mathcal I_n^*$ for every $n\in\mathbb  N$ and $A\in\mathcal I$ if and only if $A\cap A_n\in\mathcal I_n$ for each $n\in\mathbb N$.
    \begin{claim}
        The map 
        \begin{align*}
            \Psi\colon\ell_\infty&\to\ell_\infty((\ell_\infty/c_{0,\mathcal I_m})_{m\in\mathbb N})\\
            {\bf x}=(x_n)&\mapsto((x_n\chi_{A_m}(n))+c_{0,\mathcal I_m})
        \end{align*}
        is an onto Banach lattice homomorphism and $\ker\Psi=c_{0,\mathcal I}$.
    \end{claim}
In fact, from the definition, it is clear that $\Psi$ is linear and preserves the lattice operations. To check continuity, we will use  the following formula proved in \cite[Lemma 5.7]{rincon-uzcategui}. For ${\bf x}=(x_n)\in\ell_\infty$ we have
\begin{equation*}
    \|\Psi({\bf x})\|=\sup_{m\in\mathbb N}\|(x_n\chi_{A_m}(n))+c_{0,\mathcal I_m}\|=
    \sup_{m\in\mathbb N}\bigg(\mathcal I_m-\lim|x_n\chi_{A_m}(n)|\bigg)\leq\|{\bf x}\|.
\end{equation*}
    
    We will verify that $\Psi$ is onto. Let $({\bf y}^m+c_{0,\mathcal I_m})\in\ell_\infty((\ell_\infty/c_{0,\mathcal I_m})_{m\in\mathbb N})$ be given. For each $m\in\mathbb N$, let $a_m=\mathcal I_m-\lim{\bf y}^m$, which exists because $\mathcal I_m$ is maximal. Then, for every $m\in\mathbb N$, 
    \begin{equation*}
        {\bf y}^m+c_{0,\mathcal I_m}=a_m{\bf 1}+c_{0,\mathcal I_m}=a_m\chi_{A_m}+c_{0,\mathcal I_m}.
    \end{equation*}
    Define ${\bf x}=(x_n)\in\ell_\infty$ by $x_n=a_m$ whenever $n\in A_m$. Hence, 
    \begin{equation*}
        \Psi({\bf x})=((x_n\chi_{A_m}(n))+c_{0,\mathcal I_m})=((a_m\chi_{A_m}(n))+c_{0,\mathcal I_m})=({\bf y}^m+c_{0,\mathcal I_m}).
    \end{equation*}
    and therefore $\Psi$ is surjective. 

   Next, we prove that $\ker\Psi=c_{0,\mathcal I}$. Suppose that ${\bf x}=(x_n)\in\ell_\infty$ satisfies that $\Psi({\bf x})={\bf 0}$. Then $((x_n\chi_{A_m}(n))\in c_{0,\mathcal I_m})$ for each $m\in\mathbb N$. Let $\varepsilon>0$ be given. By definition, $ A(((x_n\chi_{A_m}(n)),\varepsilon)=A({\bf x},\varepsilon)\cap A_m\in\mathcal I_m$ for every $m\in\mathbb N$. By Lemma \ref{suma-directa}, it follows that $A({\bf x},\varepsilon)\in\mathcal I$. Since $\varepsilon>0$ was arbitrary, we conclude that ${\bf x}\in c_{0,\mathcal I}$.

   Finally, the map $\Psi$ induces a Banach lattice isomorphism from $\ell_\infty/c_{0,\mathcal I}$ onto $\ell_\infty((\ell_\infty/c_{0,\mathcal I_m})_{m\in\mathbb N})$. Since $\dim\ell_\infty/c_{0,\mathcal I_m}=1$ for each $m\in\mathbb N$, we have that $\ell_\infty/c_{0,\mathcal I_m}$ is Banach lattice isomorphic to $\mathbb R$. By Remark \ref{l_infty-sums}, $\ell_\infty((\ell_\infty/c_{0,\mathcal I_m})_{m\in\mathbb N})$ is Banach lattice isomorphic to $\ell_\infty$. Consequently, $\ell_\infty/c_{0,\mathcal I}$ is Banach lattice isomorphic to $\ell_\infty$.
 
   \medskip 
   
    $(3)\Longrightarrow(4)$ Suppose that  $\ell_\infty/c_{0,\mathcal I}$ is Banach lattice isomorphic to $\ell_\infty$ and let
    $T\colon\ell_\infty/c_{0,\mathcal I}\to\ell_\infty$ be a Banach lattice isomorphism. For each $n\in\mathbb N$, let ${\bf x}_n\in\ell_\infty$ be such that
    $T({\bf x}_n+c_{0,\mathcal I})=e_n$. Since ${\bf x}_n\not\in c_{0,\mathcal I}$, there are $\varepsilon_n>0$ and $A_n\not\in\mathcal I$ satisfying $\varepsilon_n\chi_{A_n}\leq|{\bf x}_n|$.
    Thus, $T(\varepsilon_n\chi_{A_n}+c_{0,\mathcal I})\leq T(|{\bf x}_n|+c_{0,\mathcal I})=|T({\bf x}_n+c_{0,\mathcal I})|=e_n$. Hence, $T(\chi_{A_n}+c_{0,\mathcal I})=a_ne_n$ for some $a_n>0$.

\begin{claim}
    If $m,n\in\mathbb N$ and $m\neq n$, then $A_n\cap A_m\in\mathcal I$.
\end{claim}
Indeed, we have $\chi_{A_m\cap A_n}+c_{0,\mathcal I}=(\chi_{A_m}+c_{0,\mathcal I})\wedge(\chi_{A_n}+c_{0,\mathcal I})$. Hence,
\begin{equation*}
    T(\chi_{A_m\cap A_n}+c_{0,\mathcal I})=T(\chi_{A_m}+c_{0,\mathcal I})\wedge T(\chi_{A_n}+c_{0,\mathcal I})=a_me_m\wedge a_ne_n={\bf0}.
\end{equation*}
    
   \begin{claim}\label{for maximality I}
       For each $n\in\mathbb N$, $\mathcal I\restriction A_n$ is maximal.
   \end{claim}
   Fix $n\in\mathbb N$. Let $B\subseteq A_n$ be such that $B\not\in\mathcal I$. Then, $T(\chi_B+c_{0,\mathcal I})\leq T(\chi_{A_n}+c_{0,\mathcal I})=a_ne_n$.
   Thus, $T(\chi_B+c_{0,\mathcal I})=b_ne_n$ for some $0<b_n\leq a_n$. Consequently, $T((\chi_{A_n}-\frac{a_n}{b_n}\chi_B)+c_{0,\mathcal I})={\bf0}$.
   It follows that $A_n\setminus B\in\mathcal I$ because of $A_n\setminus B\subset A({\bf z},1/2)$ where ${\bf z}=\chi_{A_n}-\frac{a_n}{b_n}\chi_B\in c_{0,\mathcal I}$.

   \begin{claim}\label{for maximality II}
       $B\in\mathcal I$ if, and only if, $B\cap A_n\in\mathcal I$ for each $n\in\mathbb N$.
   \end{claim}
   Suppose that $B\cap A_n\in\mathcal I$ for each $n\in\mathbb N$ and let $T(\chi_B+c_{0,\mathcal I})={\bf y}=(y_n)$. Fix $m\in\mathbb N$. Notice that ${\bf y}\geq0$, and 
   \begin{equation*}
     {\bf0}= T(\chi_{B\cap A_m}+c_{0,\mathcal I})=T(\chi_B+c_{0,\mathcal I})\wedge T(\chi_{A_m}+c_{0,\mathcal I})={\bf y}\wedge(a_me_m)=\min\{y_m,a_m\}e_m.
   \end{equation*}
   Thus, $y_m=0$. Since $m\in\mathbb N$ was arbitrary, we conclude that ${\bf y}={\bf0}$. So, $B\in\mathcal I$.
   \medskip

   Now, following the argument used in the proof of  Lemma \ref{linear independence}, we may assume that $A_m\cap A_n=\emptyset$ for every $m,n\in\mathbb N$ with $m\neq n$. 
   Define $\mathcal I_n=\mathcal I\restriction A_n\sqcup\mathcal P(A_n^c)$ for each $n\in\mathbb N$. By (1) of Lemma \ref{condicion equivalente a maximalidad} and Claim \ref{for maximality I}, $\mathcal I_n$ is maximal for all $n\in\mathbb N$. It follows from Claim \ref{for maximality II} that $\mathcal I=\bigcap_{n\in\mathbb N}\mathcal I_n$. Therefore, by Proposition \ref{particiones2}, $\mathcal I$ is strongly $\omega$-maximal. 
\end{proof}

\section{Relative complementation and  AD-families}

In this section,  we study when  $c_{0,\ideal}$ is complemented in $c_{0,\idealj}$ when $\ideal\subseteq \mathcal{J}$.
Also, we discuss the existence of $\mathcal I$-AD families and their connection with the complementation of $\mathcal I$. 

Recall that a $\mathcal{A}\subseteq \ideal^+$ is  $\ideal$-AD if $X\cap Y\in \ideal$ for all distinct  elements $X,Y\in \mathcal A$.  We let 
\begin{gather*}
\mathfrak{ad}_\ideal(\mathcal A)=\sup\{|\mathcal B|\,\colon\,\mbox{$\mathcal B\subseteq \mathcal A$ is an $\ideal$-AD family}\}.
\end{gather*}
We will write just $\mathfrak{ad}(\mathcal A)$ when it is clear from the context which ideal $\ideal$ is used. 

We begin by stating a lemma, which is a well-known consequence of Theorem \ref{talagrand theorem} (for a proof, see, for instance, \cite{Leonetti2018}).
\begin{lemma}
\label{lemma talagrand}
Let $\mathcal I$ be a meager ideal on $\mathbb N$. Then $\mathfrak{ad}(\mathcal \ideal^+)=2^{\aleph_0}$.
\end{lemma}
On the other hand,   $\mathfrak{ad}(\mathcal \ideal^+)\leq \aleph_0$ for every  $\omega$-maximal ideal $\ideal$, this was implicitly shown by  Plewik \cite{plewik} and, for the sake of completion, we include a direct proof. Let  $\ideal=\bigcap_n \ideal_n$ be an $\omega$-maximal ideal and  suppose $\mathcal{B}\subseteq \ideal^+$ is an uncountable  $\ideal$-AD family. Then,   there is $n$ such that  $|\mathcal{B}\cap \ideal_n^+|\geq 2$. If $A,B\in \mathcal{B}\cap \ideal_n^+$ are two different sets, then $A\cap B\in \ideal\subseteq \ideal_n$, which contradicts that  $\ideal_n$ is  maximal.    Additionally, any $\omega$-maximal ideal is not measurable with respect to the usual product measure on $2^\N$ (see \cite[section 4.1]{kadetsetal2022}).

The proof presented above shows that:

 \begin{corollary}(Kania \cite[Theorem A]{kania})\label{Kania result}
Let $\ideal$ be an ideal on $\N$. If $\ideal$ is complemented, then $\mathfrak{ad}(\ideal^+)\leq\aleph_0$.
\end{corollary}

The following theorem extends Kania's result to ideals $\mathcal I$ and $\mathcal J$ on $\mathbb N$ such that $\mathcal I\subsetneq \idealj$ and $c_{0,\mathcal I}$ complemented in $c_{0,\mathcal J}$. 

\begin{theorem}
\label{complemented-ad-numerable}
Let $\mathcal I\subsetneq \idealj$ be proper ideals on $\mathbb N$. Suppose that $c_{0,\mathcal I}$ is kernel of an operator from $c_{0,\mathcal J}$ to $c_{0,\mathcal J}$. Then, $\mathfrak{ad}_\ideal(\idealj\cap \ideal^+)\leq\aleph_0$. In particular, if $c_{0,\mathcal I}$ is complemented in $c_{0,\mathcal J}$, then $\mathfrak{ad}_\ideal(\idealj\cap \ideal^+)\leq\aleph_0$.
\end{theorem}

\proof
Let $Q\colon c_{0,\mathcal J}\to c_{0,\mathcal J}$ be a bounded linear operator such that $\ker Q=c_{0,\mathcal I}$. For $k,n\in \N$, let $\mathcal{A}_{n,k}$ be as in Lemma \ref{esquema2}. 
\begin{claim}
   There is $M>0$ such that $\mathfrak{ad}_\ideal(\mathcal A_{n,k})\leq kM $ for all $n,k\in\mathbb N$.
\end{claim}
  Fix $n,k\in\mathbb N$. Let $m\in\mathbb N$ be given and $A_1,\ldots,A_m\in\mathcal A_{n,k}$ be such that $A_i\cap A_j\in\mathcal I$ for $i\neq j$. If $F_1=A_1$ and $F_j=A_j\setminus(A_1\cup\cdots\cup A_{j-1})$ for $2\leq j\leq m$, then $Q(\chi_{A_j})=Q(\chi_{F_j})$ for all $1\leq j\leq m$. For $1\leq j\leq m$, let $a_j\in\mathbb K$ be such that    $a_jQ(\chi_{F_j})_n=|Q(\chi_{F_j})_n|$. Thus, 
      \begin{align*}
        \frac{m}{k}\leq\sum_{j=1}^m|Q(\chi_{A_j})_n|&=\sum_{j=1}^m|Q(\chi_{F_j})_n|\\
       &=\sum_{j=1}^ma_jQ(\chi_{F_j})_n\\
       &\leq\left\|\sum_{j=1}^ma_jQ(\chi_{F_j})\right\|\\
       &\leq\|Q\|\left\|\sum_{j=1}^ma_j\chi_{F_j}\right\|=\|Q\|.
    \end{align*}
    Hence, $m\leq k\|Q\|$.

\medskip

To end the proof, suppose $\mathcal B \subseteq \idealj\cap \ideal^+$ is an $\ideal$-AD family.  Since $\idealj\cap \ideal^+=\bigcup_{n,k}\mathcal{A}_{n,k}$ and $\mathfrak{ad}_\ideal(\mathcal{A}_{n,k})<\aleph_0$ for each $n,k\in\mathbb N$, we have that
 $\mathcal B \cap \mathcal{A}_{n,k}$ is finite for all $n,k$. Thus, $\mathcal B$ is countable, and consequently $\mathfrak{ad}_\ideal(\idealj\cap \ideal^+)\leq\aleph_0$. 
\endproof

The next observation follows immediately from Lemma \ref{lemma talagrand}.

\begin{corollary}
    Let $\mathcal I\subsetneq \idealj$ be proper ideals on $\mathbb N$ such that $\ideal\restriction A$ is meager  as a subset of $2^A$ for some infinite set $A\in \idealj$.  Then, $\mathfrak{ad}_\ideal(\idealj\cap \ideal^+)>\aleph_0$. 
\end{corollary}

An ideal $\ideal$ is {\em everywhere meager} \cite{farkas-khomskii-vidnyanszky},  if $\ideal\restriction A$ is meager  in $2^A$ for all $A\in\ideal ^+$, which in turns is equivalent to requiring that $\ideal\restriction A$ is Baire measurable   in $2^A$ for all $A\in\ideal ^+$. Every analytic ideal is everywhere meager. On the opposite side, a complemented ideal $\ideal$ is  nowhere meager, since $\mathcal I\restriction A$ is not meager in $2^A$ for all  $A\in\mathcal I^+$ by Proposition \ref{restricciones complementados}.

From Theorem \ref{complemented-ad-numerable} we get the following corollary. A similar result appears in  \cite[Corollary 1.5]{Leonetti2018} about the space $c$ of convergent sequences. 

\begin{corollary}
Let $\mathcal I\subsetneq \idealj$ be proper ideals on $\mathbb N$ such that $\ideal$ is everywhere meager. Then  $c_{0,\ideal}$ is not complemented in $c_{0,\idealj}$. In particular, $c_0$ is not complemented in $c_{0,\mathcal J}$ for any $\idealj\supsetneq\fin$.
\end{corollary}

\begin{proposition}
 Let $\mathcal I$ and $\mathcal J$ be ideals on $\N$ with $\mathcal I\subsetneq\mathcal J$ and $k$ be a positive integer. If $\mathcal J$ is complemented and $\mathcal I$ is $k$-maximal in $\mathcal J$, then $\mathcal I$ is complemented.  
\end{proposition}

\begin{proof}
    Since $\mathcal I$ is $k$-maximal in $\mathcal J$, it follows that $c_{0,\mathcal I}$ is complemented in $c_{0,\mathcal J}$ by Theorem \ref{I+particion}. Additionally, $c_{0,\mathcal J}$ is complemented in $\ell_\infty$.
    Therefore, $c_{0,\mathcal I}$ is complemented in $\ell_\infty$. This completes the proof.
\end{proof}

\begin{remark}
Concerning the above results, Theorem \ref{I+particion} gives examples of proper ideals $\mathcal I$ and $\mathcal J$ with ${\sf Fin}\subsetneq\mathcal I\subsetneq\mathcal J$ such that $c_{0,\mathcal I}$ is complemented in $c_{0,\mathcal J}$. Indeed, if $\mathcal J$ is a proper ideal and $\mathcal I_1,\ldots,\mathcal I_k$ are maximal ideals such that $\mathcal J\not\subseteq\mathcal I_j$ for all $1\leq j\leq k$, then $\mathcal I\coloneqq\mathcal J\cap(\bigcap_{1\leq j\leq k}\mathcal I_j)$ is $k$-maximal in $\mathcal J$. Thus, $c_{0,\mathcal I}$ is complemented in $c_{0,\mathcal J}$ by Theorem \ref{I+particion}. Notice that in this case $\dim(c_{0,\mathcal J}/c_{0,\mathcal I})=k$. We do not know if there exist examples with $\dim(c_{0,\mathcal J}/c_{0,\mathcal I})=\infty$.
\end{remark}

In spite of the previous results, it is still possible for \( c_{0,\mathcal{I}} \) to contain complemented subspaces that are isomorphic or isometric to \( c_0 \). In particular, we proved in \cite{rincon-uzcategui} that \( \ell_\infty(c_0) \) is isometric to \( c_{0,\fin^\omega} \), and since \( \ell_\infty(c_0) \) clearly contains a complemented copy of \( c_0 \), the same holds for \( c_{0,\fin^\omega} \). This contrasts with the situation in \( \ell_\infty \), where every subspace isomorphic or isometric to \( c_0 \) is necessarily not complemented. This leads to the following result.

\begin{proposition}\label{complemented copies of c_0 in c_0I}
    Let $\mathcal I$ be an ideal on $\mathbb N$. If $c_{0,\mathcal I}$ has a complemented copy of $c_0$, then $\mathcal I$ is not complemented.
\end{proposition}

\begin{proof}
  Suppose that $\mathcal I$ is complemented, that is, $\ell_\infty=c_{0,\mathcal I}\oplus W$. Since $c_{0,\mathcal I}$ has a complemented copy of $c_0$, there is a subspace
  $E$ isomorphic to $c_0$ such that $c_{0,\mathcal I}=E\oplus Z$. Thus, $E$ is complemented in $\ell_\infty$, which is impossible by Lindenstrauss's theorem.
\end{proof}

Now, we  present examples  of ideals that are not complemented. 
If $\mathcal I$ and $\mathcal J$ are ideals, their Fubini product $\mathcal I\times\mathcal J$ is the ideal on $\mathbb N\times\mathbb N$ defined by
\begin{gather*}
    A\in\mathcal I\times\mathcal J\quad \mbox{if and only if} \quad \{m\in\mathbb N\,\colon\,\{n\in\mathbb N\,\colon\,(m,n)\in A\}\not\in\mathcal J\}\in\mathcal I.
\end{gather*}

Regarding the Baire property of the Fubini product, in \cite{filipowetal2025} was shown that $\ideal\times\fin$ has the Baire property (hence,  it is meager) for any ideal $\ideal$. On the other hand, they also showed that $\fin\times \ideal$ has the Baire property exactly when $\ideal$ has it.

If $\mathcal A$ is a family of subsets of $\mathbb N$, the orthogonal of the family $\mathcal A$ is defined by 
$$
\mathcal A^\perp=\{B\subset\mathbb N\,\colon\,(\forall A\in\mathcal A)(B\cap A\in{\sf Fin})\}.
$$

\begin{theorem}
Let $\mathcal I$ be a proper ideal on $\mathbb N$. Then, 
\begin{enumerate}
    \item $\mathfrak{ad}({\sf Fin}\times\mathcal I)=\mathfrak{ad}(\mathcal I\times{\sf Fin})=2^{\aleph_0}$. In particular, 
    ${\sf Fin}\times\mathcal I$ and $\mathcal I\times{\sf Fin}$ are not complemented.
    \item $\mathcal I^{\omega\perp}$ is not complemented.
\end{enumerate}
\end{theorem}

\begin{proof}
(1) Let $\mathcal J\coloneqq{\sf Fin}\times\mathcal I$. We will prove that $\mathfrak{ad}(\mathcal J^+)=2^{\aleph_0}$. Let $\mathcal A$ be a ${\sf Fin}$-AD family of size $2^{\aleph_0}$. For each $A\in\mathcal A$, define $B_A=A\times\mathbb N$.
Observe that for any $m\in\mathbb N$, the section $\{n\in\mathbb N\,\colon\,(m,n)\in B_A\}$ is empty if $m\not\in A$, and equal to $\mathbb N$ if $m\in A$. Therefore,
$\{B_A\,\colon\,A\in\mathcal A\}$ forms a $\mathcal J^+$-AD family of size $2^{\aleph_0}$. In particular, it follows that ${\sf Fin}\times\mathcal I$ is not complemented by Corollary \ref{Kania result}.

To show that $\mathfrak{ad}(\mathcal I\times{\sf Fin})=2^{\aleph_0}$, we know that $\mathcal I\times{\sf Fin}$ has the Baire property by \cite[Proposition 2.6]{filipowetal2025}. Consequently, it is meager by Theorem \ref{talagrand theorem}. Therefore, we have $\mathfrak{ad}(\mathcal I\times{\sf Fin})=2^{\aleph_0}$ by Lemma \ref{lemma talagrand}. Now, from \cite{Leonetti2018} it follows that $\mathcal I\times{\sf Fin}$ is not complemented.

(2) According to  \cite[Theorem 5.4]{rincon-uzcategui}, we have the isometry $c_{0,\mathcal I^{\omega\perp}}\cong c_0(c_{0,\mathcal I^\perp})$. By the main result of \cite{cembranos}, the space $c_0(c_{0,\mathcal I^\perp})$ contains a complemented copy of $c_0$. Consequently, by Proposition \ref{complemented copies of c_0 in c_0I}, $c_{0,\mathcal I^{\omega\perp}}$ is not complemented in $\ell_\infty$.   
\end{proof}

We have the following diagram summarizing some of our results:

\[
\resizebox{\textwidth}{!}{$
\begin{array}{ccccccc}
\text{Strongly complemented} & \xRightarrow{\not\Leftarrow} & \text{Complemented + discrete projection} & \xRightarrow{\not\Leftarrow} & \text{Complemented + \eqref{Q-mult}} & \xRightarrow{\not\Leftarrow} & \text{Complemented} \\
\Updownarrow & & \Updownarrow & & \Updownarrow & & \Downarrow{\not\Uparrow} \\
\text{Strongly $\omega$-maximal} & \xRightarrow{\not\Leftarrow} & \text{Discretely represented} & \xRightarrow{\not\Leftarrow} & \text{at most }\omega\text{-maximal} & \xRightarrow{\not\Leftarrow} & \mathfrak{ad}\leq\aleph_0 \\
 & & & & & & \Downarrow{\not\Uparrow} \\
 & & & & & & \text{Non-meager}
\end{array}
$}
\]

 To conclude, we provide examples showing that all the implications in the diagram are strict. The first implication
in the lower row is strict since a strongly $\omega$-maximal ideal cannot be $k$-maximal for any $k\in\mathbb N$. Therefore, the first implication in the upper row is also strict. Example~\ref{Ejemplo-no-omega-max} shows that the second implication in the lower row is strict, and consequently, the second implication in the upper row is also strict. Moreover, there exists a complemented ideal that is not 
$\omega$-maximal \cite[Example 2]{hrusak-saenz2025}, which shows that the third implication in the upper row is strict.
For the third implication in the lower row, an example is given by \cite[Lemma 2.2 and Corollary 4.1]{kadetsetal2022}. The downward implication linking the fourth column of the two rows is strict, as shown in \cite[Example 1]{hrusak-saenz2025}. Finally, for the last downward implication, recall that
$\fin\times \ideal$ is not meager whenever $\ideal$ is a maximal ideal, but $\mathfrak{ad}({\sf Fin}\times\mathcal I)=2^{\aleph_0}$.

\section*{Acknowledgments}

We would like to thank  Alan Dow for suggesting the Example \ref{Ejemplo-no-omega-max}. We also thank for the partial support given by  the VIE-grant 4248 of Universidad Industrial de Santander.

\end{document}